\documentclass[a4paper,reqno,oneside]{amsart}

\usepackage[sort&compress,numbers]{natbib}

\usepackage[british]{babel}
\usepackage{csquotes}  
\numberwithin{equation}{section}
\usepackage{amssymb}
\usepackage[protrusion=true,expansion=true]{microtype}	
\usepackage{amsmath,amsfonts,amsthm} 
\usepackage[pdftex]{graphicx}	
\usepackage{url}
\usepackage{svg}
\usepackage[title,titletoc,page]{appendix} 
\usepackage{listings} 
\usepackage{color}
\usepackage{caption}
\usepackage{subcaption}
\usepackage{algorithm}
\usepackage{dsfont} 
\usepackage{bm}
\usepackage{upgreek}
\usepackage[normalem]{ulem}
\usepackage{enumerate}
\usepackage{mathtools}
\usepackage{lipsum}
\usepackage{adjustbox}
\usepackage{tikz}
\usepackage{circuitikz}
\usepackage{tcolorbox}
\usetikzlibrary{decorations.pathreplacing,tikzmark,calc}
\usepackage{xargs}      

\usepackage{svg}
\usepackage{nicematrix}

\usepackage[T1]{fontenc} 
\usepackage{lmodern} 
\rmfamily 
\DeclareFontShape{T1}{lmr}{b}{sc}{<->ssub*cmr/bx/sc}{}
\DeclareFontShape{T1}{lmr}{bx}{sc}{<->ssub*cmr/bx/sc}{}

\usepackage[left=2.4cm,right=2.4cm,top=3.0cm,bottom=3.0cm,footskip=1.5cm,headsep=1cm]{geometry}
\usepackage[%
pdfauthor={Yannick De Bruijn, Erik Orvehed Hiltunen},%
pdftitle={PHOTONIC BANDGAPS III},
pdfsubject={PHOTONIC BANDGAPS III},
pdfkeywords={Subwavelength resonances, evanescent modes, band gap, interface eigenmodes, Bloch theory, complex Brillouin zone, layer potentials}]{hyperref}

\usepackage{fancyhdr}
\fancypagestyle{plain}{%
\fancyhead[C]{}
\fancyfoot[C]{}

}

\fancypagestyle{nosection}{%
\fancyhead[CE]{}
\fancyhead[CO]{}
\fancyfoot[CE,CO]{\thepage}
}
\numberwithin{equation}{section}		
\numberwithin{figure}{section}			
\numberwithin{table}{section}			

\newcommand{\vect}[1]{\boldsymbol{\mathbf{#1}}}

\newcommand{\Cb}{\mathbb{C}}
\DeclareMathOperator{\C}{\mathbb{C}}

\newcommand{\F}{\mathcal{F}}

\newcommand{\N}{\mathbb{N}}

\newcommand{\R}{\mathbb{R}}

\newcommand{\I}{\mathcal{I}}

\newcommand{\Z}{\mathbb{Z}}

\renewcommand{\epsilon}{\varepsilon}
\newcommand{\dx}{\: \mathrm{d}}

\newcommand{\uf}{\mathfrak{u}}
\newcommand{\Bf}{\mathfrak{B}}
\newcommand{\Cf}{\mathfrak{C}}

\renewcommand{\Re}{\mathrm{Re}}

\newcommand{\iu}{\mathrm{i}\mkern1mu}

\renewcommand{\tilde}{\widetilde}

\renewcommand{\i}{\mathrm{i}}
\renewcommand{\d}{\,\mathrm{d}}
\renewcommand{\Re}{\mathfrak{Re}}
\renewcommand{\Im}{\mathfrak{Im}}

\newcommand{\iL}{\mathsf{L}}
\newcommand{\Id}{\mathrm{Id}}
\newcommand{\iR}{\mathsf{R}}

\makeatletter
\DeclareRobustCommand\vdots{%
  \mathpalette\@vdots{}%
}
\newcommand*{\@vdots}[2]{%
  \sbox0{$#1\cdotp\cdotp\cdotp\m@th$}%
  \sbox2{$#1.\m@th$}%
  \vbox{%
    \dimen@=\wd0 %
    \advance\dimen@ -3\ht2 %
    \kern.5\dimen@
    \dimen@=\wd2 %
    \advance\dimen@ -\ht2 %
    \dimen2=\wd0 %
    \advance\dimen2 -\dimen@
    \vbox to \dimen2{%
      \offinterlineskip
      \copy2 \vfill\copy2 \vfill\copy2 %
    }%
  }%
}
\DeclareRobustCommand\ddots{%
  \mathinner{%
    \mathpalette\@ddots{}%
    \mkern\thinmuskip
  }%
}
\newcommand*{\@ddots}[2]{%
  \sbox0{$#1\cdotp\cdotp\cdotp\m@th$}%
  \sbox2{$#1.\m@th$}%
  \vbox{%
    \dimen@=\wd0 %
    \advance\dimen@ -3\ht2 %
    \kern.5\dimen@
    \dimen@=\wd2 %
    \advance\dimen@ -\ht2 %
    \dimen2=\wd0 %
    \advance\dimen2 -\dimen@
    \vbox to \dimen2{%
      \offinterlineskip
      \hbox{$#1\mathpunct{.}\m@th$}%
      \vfill
      \hbox{$#1\mathpunct{\kern\wd2}\mathpunct{.}\m@th$}%
      \vfill
      \hbox{$#1\mathpunct{\kern\wd2}\mathpunct{\kern\wd2}\mathpunct{.}\m@th$}%
    }%
  }%
}
\makeatother

\newcommand{\arccosh}[1]{\operatorname{arccosh}\left(#1\right)}

\usepackage[colorinlistoftodos]{todonotes}

\usepackage[noabbrev,capitalize]{cleveref}
\crefname{proposition}{Proposition}{Propositions}
\crefname{equation}{}{}

\newtheorem{theorem}{Theorem}[section]
\newtheorem{lemma}[theorem]{Lemma}
\newtheorem{proposition}[theorem]{Proposition}
\newtheorem{corollary}[theorem]{Corollary}

\theoremstyle{definition}
\newtheorem{definition}[theorem]{Definition}

\crefname{assumption}{Assumption}{Assumptions}
\crefname{definition}{Definition}{Definitions}
\crefname{corollary}{Corollary}{Corollaries}
\crefname{enumi}{item}{items}

\begin{document}
\title[Complex Band Structure for tridiagonal $k$-Toeplitz operators and localisation transition in defected finite non-Hermitian systems]{Complex Band Structure and localisation transition for tridiagonal non-Hermitian $k$-Toeplitz operators with defects}

\author[Y. De Bruijn]{Yannick De Bruijn}
\address{\parbox{\linewidth}{Yannick De Bruijn\\
Department of Mathematics, ETH Z\"urich, Rämistrasse 101, 8092 Z\"urich, Switzerland.}}
\email{ydebruijn@student.ethz.ch}

\author[E. O. Hiltunen]{Erik Orvehed Hiltunen}
\address{\parbox{\linewidth}{Erik Orvehed Hiltunen\\
Department of Mathematics, University of Oslo, Moltke Moes vei 35, 0851 Oslo, Norway.}}
\email{erikhilt@math.uio.no}

\begin{abstract}
    Using the Bloch-Floquet theory, we propose an innovative technique to obtain the eigenvectors of tridiagonal $k$-Toeplitz operators. This method offers a more extensive and quantitative basis for describing localised eigenvectors beyond the non-trivial winding zone, yielding sharp decay bounds. The validity of our results is confirmed numerically in one-dimensional resonator chains, showcasing non-Hermitian skin localisation, bulk localisation, and tunnelling effects. We conclude the paper by analysing non-Hermitian tight binding Hamiltonians, illustrating the broad applicability of the complex band structure.
\end{abstract}
\maketitle
\vspace{3mm}
\noindent
\textbf{Mathematics Subject Classification (MSC2010): }15A18 
15B05, 
81Q12.  

\vspace{3mm}

\vspace{5mm}

\section{Introduction}

We introduce the complex band structure as a novel tool for constructing eigenvectors of tridiagonal $k$-Toeplitz operators, enhancing previous methods such as the non-trivial winding region \cite{bottcher.silbermann1999Introduction, gohberg1993basic, REICHEL1992153, trefethen.embree2005Spectra, ammari2024spectra}. This is particularly relevant because many discretised one-dimensional physical systems and models such as the tight-binding and nearest-neighbour approximations admit tridiagonal $k$-Toeplitz representations. In contrast to topological invariants, such as the winding number, the complex band structure offers quantitative decay estimates for eigenvectors, allowing the investigation of localisation transitions amid imperfections and disorder \cite{PhysRevResearch.1.023013, PhysRevB.56.R4333, PhysRevLett.124.056802}.
As the spectra of non-Hermitian matrices are highly sensitive to perturbations, we revert to pseudospectra to construct eigenvectors for finite tridiagonal $k$-Toeplitz matrices \cite{REICHEL1992153, trefethen.embree2005Spectra}.

The non-Hermitian skin effect refers to the localisation of a significant portion of bulk eigenmodes at the boundary of a non-Hermitian open system \cite{ashida.gong.ea2020NonHermitian, PhysRevLett.121.086803, PhysRevLett.125.126402, PhysRevLett.124.086801}. It has been experimentally demonstrated in fields such as topological photonics and phononics \cite{ghatak.brandenbourger.ea2020Observation, PhysRevLett.124.086801, longhi}. This effect is crucial for the progression of active metamaterials, offering innovative methods to direct energy on subwavelength scales \cite{ammari2024functional, feppon.cheng.ea2023Subwavelength}.

We illustrate our findings in subwavelength defected non-Hermitian resonator chains and present analogous results to \cite{debruijn2025complexbrillouinzonelocalised, debruijn2024complexbandstructuresubwavelength} in the non-Hermitian case. We show that the complex band structure facilitates the creation of $\varepsilon$-pseudoeigenvectors which captures the quantitive decay properties in finite systems. This is further motivated by the fact that the Bloch spectrum is known to be the limiting spectrum of a finite but large resonator chain \cite{ammari2024generalisedbrillouinzonenonreciprocal, 10.1121/10.0022535} where, in the non-Hermitian case, the Floquet-Bloch quasimomentum must be complex. We generalise this idea by allowing a general complex quasimomentum, resulting in branches of the Bloch band structure which describes localised defect modes. We show that depending on the type of defect, we may observe a transition from  skin (edge) localisation to bulk localisation \cite{PhysRevB.111.035109, NonHermitianCloaking}. This transition may be explained by a non-trivial winding region of the Toeplitz operator, which provides a qualitatve estimate for skin localisation \cite{ammari2024spectra, PhysRevLett.121.086803,PhysRevLett.125.126402}. The complex band structure is also able to concicely account for this transition, and additionally provides explicit and sharp estimates of the localisation length of the defect modes. 

We demonstrate that the complex band structure for tridiagonal $k$-Toeplitz matrices is instrumental across various non-Hermitian physical settings \cite{PhysRevLett.125.226402, PhysRevLett.121.086803, Zhang_2024, Non-HermitianTopologicalPhenomenaAReview, PhysRevLett.124.086801, PhysRevLett.125.126402, PhysRevLett.121.086803}. Notably, the localisation transition outlined in \cite{PhysRevB.111.035109} is fully described by the complex band structure. This framework enables in-depth exploration of tight binding Hamiltonians with non-reciprocal coupling, as introduced by Hatano and Nelson \cite{PhysRevB.58.8384, PhysRevLett.77.570}.

This paper is organised as follows. In Section \ref{sec: tridiagonal k toeplitz}, we demonstrate how the complex band structure can be used to construct eigenvectors of tridiagonal $k$-Toeplitz operators and that truncated eigenvectors form exponentially good pseudoeigenvectors. In Section \ref{Sec: quasiperiodic Gauge capacitance} we generalise the quasiperiodic gauge capacitance matrix to complex quasimomenta and establish a spectral convergence result of the finite case to the quasiperiodic case. In Section \ref{Sec: Defected finite resonator chains}, we present our finding on finite resonator chains and illustrate that the complex band structure accurately captures the localisation properties in defected non-Hermitian resonator chains. Moreover, we explore the possibility to achieve bulk localisation  in defected non-Hermitian resonator chains. In Section \ref{sec: Non-Hermitian Hamiltonian}, we demonstrate how the complex band structure of tridiagonal $k$-Toeplitz operators can address a wide range of non-Hermitian problems, specifically highlighting its application to a non-Hermitian tight-binding Hamiltonian. In Section \ref{Sec: concluding remarks} we make
some concluding remarks and give some possible extensions of our present work. The \texttt{Matlab} code that supports the findings of this article is openly available in Section \ref{Sec: Data availability}.

\section{Complex Band structure for tridiagonal $k$-Toeplitz Operators}\label{sec: tridiagonal k toeplitz}

In the present section, we introduce the complex band structure, in a general setting of  tridiagonal $k$-Toeplitz operators, 
and how it facilitates the construction of  eigenvectors of the system. The subsequent approach presented here addresses the limitations of the qualitative winding-region estimates by furnishing explicit decay rates and characterising the eigenvectors across an expanded spectrum.

We define a \emph{tridiagonal $k$-Toeplitz operator} by
\begin{equation}\label{equ:tridiagonalktoeplitzop1}
    \vect{A} := \left(\begin{smallmatrix}
        a_1 & b_1 & 0 & & & & &  \\
        c_1 & a_2 & b_2 &  \ddots &  &\\
        0  & \ddots &  \ddots & \ddots  & \ddots & \\
         & \ddots& c_{k-2} & a_{k-1}& b_{k-1} &\ddots& \\
        & & \ddots & c_{k-1} & a_k  & b_k& \ddots&\\
        & & & \ddots & c_{k} & a_1  & b_1& \ddots&\\
        & & & & \ddots & \ddots & \ddots& \ddots &\ddots \\
    \end{smallmatrix}\right),
\end{equation}
where $a_i,b_i,c_i\in\R$ for $i\in\N$ and $\vect{A}_{ij}=0$ if $\vert i-j \vert > 1$ for $i,j\in\N$. Throughout this section, we posit $\prod_{i=1}^k \frac{b_i}{c_i} > 1$. If not, analogous conclusions follow using the adjoint operator $\vect{A}^*$.

Any tridiagonal $k$-Toeplitz operator $\vect{A}$ can be reformulated as a tridiagonal block Toeplitz operator with $k \times k$ blocks exhibiting a $1$-periodic repetition,
\begin{equation}\label{equ:tridiagonalktoeplitzop2}
    \vect{A} = \left(\begin{smallmatrix}
        \vect{A}_0 & \vect{A}_{-1} &  \\
        \vect{A}_{1} & \vect{A}_0 & \ddots  \\
         & \ddots & \ddots   \\
    \end{smallmatrix}\right).
\end{equation}

The \emph{symbol} of a tridiagonal block Toeplitz operator \eqref{equ:tridiagonalktoeplitzop2}, and thus also of a tridiagonal $k$-Toeplitz operator \eqref{equ:tridiagonalktoeplitzop1}
is defined by the matrix valued function
\begin{align}\label{def: symbol of operator}
    f:\mathbb C &\to \mathbb C^{k\times k}\nonumber\\
    z&\mapsto \vect{A}_{-1}z^{-1} + \vect{A}_0 + \vect{A}_1z.
\end{align}

In the sequel of this paper, we shall refer to the Toeplitz operator $\vect{A}$, specified in equation \eqref{equ:tridiagonalktoeplitzop2}, by the notation $\vect{A} = \vect{T}(f)$. To define a finite Toeplitz matrix, we introduce the projection
\begin{align}
    \vect{P}_n: \ell^2(\N, \mathbb{C}) &\to \ell^2(\N, \mathbb{C})\\
    (x_1, x_2, x_3, \dots) &\mapsto (x_1, \dots, x_n, 0, 0, \dots).
\end{align}
A tridiagonal $k$-Toeplitz matrix is derived from a tridiagonal $k$-Toeplitz operator by performing the following truncation,
\begin{equation}
    \vect{T}_{mk}(f) := \vect{P}_{mk} \vect{T}(f) \vect{P}_{mk}.
\end{equation}

\subsection{Complex band structure}
It has been established that, under the non-degeneracy condition $b_ic_i > 0$, every tridiagonal matrix with real-valued entries is similar to a symmetric matrix. This transformation is noteworthy because it ensures real spectra, which is why such a class of problems is often termed pseudo-Hermitian \cite{ PhysRevLett.77.570, PhysRevLett.80.5243, Fern_ndez_2022}. 
Therefore, in subsequent discussions, our attention will be concentrated on spectra that are real-valued. In \cite{ammari2024generalisedbrillouinzonenonreciprocal} various limiting spectra for tridiagonal matrices have been presented. In particular, the physically important limit consists of an \emph{open} system corresponding to the limit of a finite but large Toeplitz matrix.

We introduce the parameter $L$, representing the length of the lattice vector $\Lambda$ spanning the unit cell $Y$, and define the first Brillouin zone as $Y^* = (-\pi / L, \pi/L]$.
The spectrum of this open limit is characterised in a Floquet-Bloch-type way with a generalised Brillouin zone.

\begin{theorem}\cite[Theorem 4.1.]{ammari2024generalisedbrillouinzonenonreciprocal}\label{Thm: spectral limit CBZ}
    Let $\vect{T}(f)$ be the tridiagonal Toeplitz operator with symbol $f(z) \in \R^{k \times k}$ and $b_ic_i > 0$ for all $1 \leq i \leq k$. We then have
    \begin{equation}\label{eq. spectral theorem limit}
        \lim_{m \to \infty}\sigma\bigl(\vect{T}_{mk}(f)\bigr) = \bigcup_{\alpha \in Y^*} \sigma\bigl(f(e^{-\i(\alpha + \i r)L})\bigr),
    \end{equation}
    where $r = \frac{1}{2}\log \Bigl(\prod_{j = 1}^k \frac{b_j}{c_j}\Bigr)$.
\end{theorem}

For Hermitian matrices, we note that $r=0$, while for non-Hermitian problems, \eqref{eq. spectral theorem limit} can be thought of as a generalisation of the standard Floquet-Bloch band theory. For the remainder of this section, and without loss of generality, let us assume $L = 1$. For later use we define the spectrum of the open limit \eqref{eq. spectral theorem limit} as follows,
\begin{equation}
    \sigma_{\mathrm{open}} = \bigcup_{\alpha \in Y^*} \sigma\bigl(f(e^{-\i(\alpha + \i  r)})\bigr).
\end{equation}
In the same way as in \cite{debruijn2024complexbandstructuresubwavelength}, we seek to generalise the Floquet spectrum by allowing a general imaginary part of the complex quasimomenta. As we shall see, this will also encompass gap eigenvalues arising from compact defects of the Toeplitz operator.
\begin{definition}\label{def: Band function}
    For the Toeplitz operator $\vect{T}(f)$, the complex band functions are the real-valued eigenvalues of the symbol function $f(e^{-\i(\alpha + \i \beta)})$ and are parametrised in both $\alpha$ and $\beta$.
\end{definition}

We start by presenting a general symmetry result that defines the properties of the real and imaginary components of the quasimomentum.

\begin{theorem}\label{Thm: alpha and beta fixed}
    Let $\vect{T}(f)$ be a tridiagonal $k$-Toeplitz operator with symbol $f \in \R^{k\times k}$ and $b_i c_i > 0$ for $1 \leq i \leq k$. Given a real eigenvalue $\lambda \in \R$ of $f(e^{-\i(\alpha + \i \beta)})$, it holds that either $\alpha \in \{0, \pi\}$ or $\beta = \frac{1}{2}\log\bigl( \prod_{i = 1}^k \frac{b_i}{c_i} \bigr)$.
\end{theorem}

\begin{proof}
    We seek real roots $\lambda$ of the characteristic equation $\operatorname{det}\bigl(f(e^{-\i(\alpha + \i \beta)})- \lambda \Id\bigr) = 0$. From \cite[Lemma 2.6.]{ammari2024spectra}, the characteristic polynomial admits the following expansion
    \begin{equation}\label{eq: reduced characteristic function}
        \operatorname{det}\bigl(f(e^{-\i(\alpha + \i \beta)})- \lambda \Id\bigr) = A e^{-\i(\alpha + \i \beta)} + Be^{\i(\alpha + \i \beta)} + g(\lambda),
    \end{equation}
    where $A = (-1)^{k +1} \prod_{i =1}^k c_i$,  $B = (-1)^{k+1} \prod_{i =1}^k b_i$ and $g(\lambda)$ is a polynomial with real valued coefficients given by \begin{equation}\label{equ:defiglambda1}
        g(\lambda) = \det(\vect{A}_0-\lambda)  -b_kc_k p(\lambda)  ,
    \end{equation}
    where
    \begin{equation}
    p(\lambda) = \begin{cases}
                    0, & k=1, \\
                    1, & k=2,\\
                    \left\lvert\begin{smallmatrix}
                            a_2 - \lambda & b_2 & 0 & \dots & 0 \\
                            c_3 & a_3 - \lambda & b_3 & \ddots & \vdots \\
                            0 & \ddots & \ddots & \ddots & 0 \\
                            \vdots & \ddots & \ddots & \ddots & b_{k-2} \\
                            0 & \dots & 0 & c_{k-1} & a_{k-1} - \lambda \\
                    \end{smallmatrix}\right\rvert, & k \geq 3.
                \end{cases}
    \end{equation}
    Therefore, since $\lambda \in \R$, it follows that $g(\lambda) \in \R$. Hence, it must hold that 
    \begin{equation}\label{eq:condition}
        A e^{-\i(\alpha + \i \beta)} + Be^{\i(\alpha + \i \beta)} \in \R.
    \end{equation}
    Let us proceed by the change of variables $\beta = r + \tilde{\beta}$, which gives
    \begin{equation}\label{eq: eigval polynomial}
        A e^{-\i(\alpha + \i \beta)} + Be^{\i(\alpha + \i \beta)} = A e^{r}e^{\tilde{\beta}} e^{-\i\alpha} + B e^{-r}e^{-\tilde{\beta}}e^{\i\alpha}.
    \end{equation}
    By choosing $r$ such that the coefficients $Ae^{r} = Be^{-r}$ are equal, which holds for,
    \begin{equation}\label{eq: non reciproctity rate}
        r =  \frac{1}{2}\log\Bigl( \prod_{i =1}^k \frac{b_i}{c_i} \Bigr),
    \end{equation}
    the expression in \eqref{eq: eigval polynomial} simplifies to
    \begin{equation}
        Ae^{r} \Bigl (e^{-\i(\alpha + \i \beta)} + e^{\i(\alpha + \i\beta)}\Bigr)= 2Ae^{r} \bigl(\cos(\alpha)\cosh(\tilde{\beta}) +\i\sin(\alpha)\sinh(\tilde{\beta}) \bigr).
        \end{equation}
        For the imaginary part to vanish it must hold that $\alpha = \{0, \pi\}$ or $\tilde{\beta}= 0$, and from \eqref{eq: reduced characteristic function} we conclude that 
        \begin{equation}\label{eq: eigval of symbol}
        \sigma\bigl(f(e^{-\i(\alpha + \i \beta)})\bigr) = \Bigl\{ \lambda \in \R~|~2Ae^{r} \cos(\alpha)\cosh(\tilde{\beta}) + g(\lambda) = 0 \Bigr\},
    \end{equation}
    which completes the proof.
\end{proof}

Theorems \ref{Thm: alpha and beta fixed} and \ref{Thm: spectral limit CBZ} allow us to distinguish between two different families of the complex band functions expressed in Definition \ref{def: Band function}.

\begin{definition}
    \label{def: reduced band functions}
    For the tridiagonal $k$-Toeplitz operator $\vect{T}(f)$ with symbol $f \in \R^{k\times k}$ and $b_i c_i > 0$ for $1 \leq i \leq k$, we define the \emph{band functions}, parametrised by $\alpha$, as the $k$ eigenvalues of the symbol $f(e^{-\i(\alpha + \i \beta)})$ for fixed $\beta = \frac{1}{2}\log\bigl( \prod_{i =1}^k \frac{b_i}{c_i} \bigr) $. We define the \emph{gap functions} analogously, parametrised by $\beta$ for fixed $\alpha \in \{0, \pi\}$. 
\end{definition}

An example of band and gap functions for a tridiagonal $k$-Toeplitz matrix will be treated in the following section and we refer the reader to Figure \ref{fig: Monomer Band function surface}, Figure \ref{Fig:Gauge_band} and Figure \ref{fig: Band functions non-Hermitian Hamiltonian} for a visualisation. We would like to emphasise that this treatment of complex band structure generalises the definition given in \cite{debruijn2024complexbandstructuresubwavelength} as it allows us to treat non-Hermitian operators. 

We proceed to highlight yet another symmetry of the symbol of a tridiagonal $k$-Toeplitz operator.

\begin{corollary}\label{cor: same eigenvalue}
    Let $\lambda \in \mathbb{R}$ be an eigenvalue of $f\bigl(e^{-\i(\alpha + \i(r + \tilde{\beta}))}\bigr)$ then $\lambda$ is also an eigenvalue of 
    $f\bigl(e^{-\i(-\alpha + \i(r + \tilde{\beta}))}\bigr)$ or of $f\bigl(e^{-\i(\alpha + \i(r - \tilde{\beta}))}\bigr)$, where $r$ is defined as in \eqref{eq: non reciproctity rate}.
\end{corollary}
\begin{proof}
    Let $\lambda$ be an eigenvalue of the symbol $f\bigl(e^{-\i(\alpha + \i(r + \tilde{\beta}))}\bigr)$ then by \eqref{eq: eigval of symbol} it must hold that
    \begin{equation}
       2Ae^{r} \cos(\alpha)\cosh(\tilde{\beta}) + g(\lambda) = 0.
    \end{equation}
    Due to the symmetry $\cos(\alpha) = \cos(-\alpha)$ and $\cosh(\tilde{\beta}) = \cosh(-\tilde{\beta})$, and the assertion follows.
\end{proof}

Our next goal is to demonstrate how the complex band structure may be used to infer properties of the eigenvectors of the Toeplitz operator $\vect{T}(f)$.
Following \cite[Theorem 3.7]{ammari2024generalisedbrillouinzonenonreciprocal}, a linear combination of the quasiperiodically extended eigenvectors of the symbol may be used to construct an exact eigenvector of the Toeplitz operator $\vect{T}(f)$. For completeness, we recall this construction, as in the presented approach we also take into account the complex quasiperiodicity.

Let $\vect{v}\in \R^{k\times 1}$ be an eigenvector of $f(e^{-\i(\alpha + \i\beta)})\in \R^{k\times k}$ with eigenvalue $\lambda$, then
    the vector
    \begin{equation}
        \vect{u}_1 = \bigl(\vect{v}^\top, e^{\i(\alpha + \i\beta)}\vect{v}^\top, e^{2\i(\alpha + \i\beta)}\vect{v}^\top, \dots \bigr),
    \end{equation}
    satisfies the eigenvalue problem $\vect{T}(f) \vect{u}_1 = \lambda \vect{u}_1$ everywhere except in the first row.
    By Corollary \ref{cor: same eigenvalue} we may find a $\vect{w}$ which is either an eigenvector of $f\bigl(e^{-\i(-\alpha + \i(r + \tilde{\beta}))}\bigr)$ or of $f\bigl(e^{-\i(\alpha + \i(r - \tilde{\beta}))}\bigr)$ which is linearly independent of $\vect{v}$. In the case where the two eigenvectors become confluent, that is at the edge of the open spectrum, a linearly independent eigenvector may be recovered by the same method as in \cite[Theorem 2.10.]{ammari2024spectra}. Suppose $\vect{w}$ is an eigenvector of $f\bigl(e^{-\i(\mp\alpha + \i(r \pm \tilde{\beta}))}\bigr)$, then by a quasiperiodic extension one may find
    \begin{equation}
        \vect{u}_2 = \bigl(\vect{w}^\top, e^{\i(\mp\alpha + \i(r \pm \tilde{\beta}))}\vect{w}^\top, e^{2\i(\mp\alpha + \i(r \pm \tilde{\beta}))}\vect{w}^\top, \dots \bigr),
    \end{equation}
    which is linearly independent of $\vect{u}_1$ and satisfies the eigenvalue problem
    $\vect{T}(f) \vect{u}_2 = \lambda \vect{u}_2$ everywhere except in the first row. A linear combination of the two vectors $\vect{u}_1$ and $\vect{u}_2$ is therefore an exact eigenvector of $\vect{T}(f)$.
    Explicit decay rates or decay bounds of an eigenvector of $\vect{T}(f)$ may be achieved as follows.

\begin{corollary}\label{cor: decay rate spectrum}
    Let $\lambda\in \sigma_{\mathrm{open}}$ be an eigenvalue in the Floquet spectrum. Let $\vect{u}$ be an eigenvector of a tridiagonal $k$-Toeplitz operator $\vect{T}(f)$ with symbol $f \in \R^{k\times k}$ and $b_i c_i > 0$ for $1 \leq i \leq k$, then the eigenvector entries decay exponentially, that is
    \begin{equation}\label{eq:decay_inside}
        \frac{|\vect{u}^{(i + k)}|}{|\vect{u}^{(i)}|} = e^{-\beta},
    \end{equation}
    where $\beta = \frac{1}{2}\log\bigl( \prod_{i =1}^k \frac{b_i}{c_i} \bigr)$.
\end{corollary}

\begin{proof}
    For an eigenvalue in the open spectrum, it follows from Theorem \ref{Thm: alpha and beta fixed} that $\beta = \frac{1}{2}\log\bigl( \prod_{i =1}^k \frac{b_i}{c_i} \bigr)$ is fixed. So, if $\lambda$ is an eigenvalue of $f(e^{-\i(\alpha + \i \beta)})$ with the eigenvector $\vect{v}\in \R^{k\times 1}$ then $\lambda$ is also an eigenvalue of $f(e^{-\i(-\alpha + \i \beta)})$, with eigenvector $\vect{w}\in \R^{k\times 1}$. Suppose that $\vect{v}$ is linearly independent of $\vect{w}$, if not, a linearly independent eigenvector may be recovered by the same techniques as present in \cite{ammari2024spectra}. An eigenvector of the Toeplitz operator may be constructed by the right linear combination of $\vect{v}$ and $\vect{w}$,
    \begin{equation}
        \vect{u} = \Bigl( (a\vect{v}^\top + b \vect{w}^\top), e^{-\beta}(a\vect{v}^\top e^{-\i\alpha} + b \vect{w}^\top e^{\i\alpha}), e^{-2\beta}(a\vect{v}^\top e^{-2\i\alpha} + b \vect{w}^\top e^{2\i\alpha}), \dots\Bigr),
    \end{equation}
    for some coefficients $a, b \in \R$.
    The assertion follows directly by the construction of the eigenvector $\vect{u}$.
\end{proof}

We emphasise that \eqref{eq:decay_inside} no longer holds for eigenvalues which do not lie within $\sigma_\mathrm{open}$. In this case, one can obtain a similar decay bound on the entries of the eigenvectors which we present in Corollary \eqref{cor: decay rate outside spectrum} below.

We now define the region of non-trivial winding. We let $\mathbb{T}$ denote the unit circle in $\mathbb{C}$, that is $\mathbb{T} = e^{-\i[0, 2\pi]}$, and consider the winding region limited to real spectral values,
\begin{equation}\label{eq: def region of positive winding}
    \sigma_{\mathrm{wind}} := \bigl\{ \lambda \in \R~|~ \operatorname{wind}\bigl(\operatorname{det}(f(\mathbb{T}) - \lambda\Id) , 0 \bigr) > 0 \bigr\},
\end{equation}
as well as the edge of the winding region
\begin{equation}
    \sigma_{\mathrm{det}} := \bigl\{ \lambda \in \R~|~\operatorname{det}\bigl(f(\mathbb{T})-\lambda \Id\bigr) = 0 \bigr\}.
\end{equation}
\begin{lemma}\label{lemma: inclusion spectrum winding region}
    The following inclusion
    \begin{equation}
        \sigma_{\mathrm{open}} \subseteq \sigma_{\mathrm{wind}}
    \end{equation}
    is always satisfied.
\end{lemma}
\begin{proof}
    Let $\lambda \in \sigma_\text{open}$, then there exists an $\alpha^* \in Y^*$ such that
    \begin{equation}
     \operatorname{det}\bigl(f(e^{-\i(\alpha^* + \i r)})- \lambda \Id\bigr) = 0.
    \end{equation}
    By \cite[Lemma 2.6.]{ammari2024spectra} this is equivalent to
    \begin{equation}
        \psi(e^{-\i(\alpha^* + \i r)}) + g(\lambda) = 0,
    \end{equation}
    where $\psi(z) = (-1)^{k+1} z \prod_{i = 1}^k c_i + (-1)^{k+1} z^{-1} \prod_{i = 1}^k b_i$.
    In other words $\lambda\in \sigma_\text{open}$ if and only if $-g(\lambda) = 
    \psi(e^{-\i(\alpha^* + \i r)}) $. 
    Let us denote by $E_\beta$ the ellipse with interior, traced by $\psi(e^{-\i([0, 2\pi] + \i \beta)})$, then by \cite[Lemma 3.6]{ammari2024generalisedbrillouinzonenonreciprocal}  it follows that $E_{\beta_2} \subset E_{\beta_1}$ for $0 \leq \beta_1 < \beta_2 \leq r$, where $r$ was defined in \eqref{eq: non reciproctity rate}.
    Recall from \cite{ammari2024spectra} that the positive winding region can alternatively to \eqref{eq: def region of positive winding} be defined as
    \begin{equation}
        \sigma_\text{wind} = \bigl\{ \lambda\in \mathbb{C} ~|~ \operatorname{wind}\bigl(\psi(\mathbb{T}), -g(\lambda)\bigr) > 0 \bigr\}.
    \end{equation}
    Therefore it holds that $-g(\lambda) \in E_r \subset E_0$ from which it follows that $\lambda \in \sigma_\text{wind}$.
\end{proof}

We also define the complement of the  winding region as
\begin{equation}
    \sigma_{\mathrm{wind}}^\mathsf{c} := \R \setminus \sigma_{\mathrm{wind}} .
\end{equation}

For an eigenvalue outside the Floquet spectrum one can still achieve an exponential decay estimate for a Toeplitz operator though not as sharp as the previous result.

\begin{corollary}\label{cor: Toeplitz decay rate outside spectrum}
   Let $\lambda \in \sigma_{\mathrm{wind}} \setminus \sigma_{\mathrm{open}}$ be an eigenvalue outside the Floquet spectrum. Let $\vect{u}$ be an eigenvector of a tridiagonal $k$-Toeplitz operator $\vect{T}(f)$ with symbol $f \in \R^{k\times k}$ and $b_i c_i > 0$ for $1 \leq i \leq k$, then the eigenvector entries decay exponentially, that is
    \begin{equation}\label{eq:decay_outside}
        \frac{|\vect{u}^{(i + k)}|}{|\vect{u}^{(i)}|} = \mathcal{O}\bigl( e^{-(r-\tilde{\beta})}\bigr),
    \end{equation}
    where $r=\frac{1}{2}\log\left( \prod_{i =1}^k \frac{b_i}{c_i} \right)$ is fixed and $\tilde{\beta}(\lambda) = \arccosh{-\frac{g(\lambda)}{2Ae^{r}}}$ depends on the eigenvalue $\lambda$ and where we choose the positive root of the $\arccosh{\cdot}$.
\end{corollary}

\begin{proof}
    For an eigenvalue outside the Floquet spectrum, by Theorem \ref{Thm: alpha and beta fixed} we know that $\alpha = \{0, \pi \}$ is fixed. So an eigenvector of the Toeplitz operator may be constructed in a similar fashion as in the poof of Theorem \ref{cor: decay rate spectrum} by the right linear combination,
    \begin{equation}
        \vect{u} = \Bigl( (a\vect{v}^\top + b \vect{w}^\top), e^{\i \alpha}(a e^{-(r+\tilde{\beta})}\vect{v}^\top + b e^{-(r-\tilde{\beta})}\vect{w}^\top), e^{2\i \alpha}(a e^{-2(r+\tilde{\beta})}\vect{v}^\top + b e^{-2(r-\tilde{\beta})}\vect{w}^\top), \dots \Bigr).
    \end{equation}
    By the construction of the eigenvector $\vect{u}$, we have $\frac{|\vect{u}^{(i + k)}|}{|\vect{u}^{(i)}|} = \mathcal{O}(e^{-B}$), where $B = \min\{r + \tilde{\beta}, r-\tilde{\beta}\} = r-\tilde{\beta}$ and for $\lambda\in \sigma_\text{wind}$ it follows $ B > 0$.
    The explicit value of $\tilde{\beta}$ is achieved by solving \eqref{eq: eigval of symbol} for $\tilde{\beta}$ and therefore the assertion follows.
\end{proof}

Frequencies outside the spectrum can only be achieved through the introduction of defects, such as by disrupting the repetition pattern of the tridiagonal Toeplitz matrix. The translational symmetry is disrupted by this defect, marking the defect site as the origin of the system. The system is then represented by a doubly infinite $k$-block-tridiagonal matrix $\vect{L}(f)$, known as a \emph{Laurent operator}, with perturbed entries at the row of index $i=m$.

\begin{corollary}\label{cor: decay rate outside spectrum}
    Let $\lambda\in \sigma_\text{wind}^\mathsf{c}$ and let $m$ be the defect site. Let $\vect{u}$ be an eigenvector of a defected tridiagonal $k$-Laurent operator $\vect{L}(f)$ with symbol $f \in \R^{k\times k}$ and $b_i c_i > 0$ for $1 \leq i \leq k$, then the eigenvector entries decay exponentially.
    \begin{itemize}
        \item Leading up to the defect site, that is for $i \leq m$
        \begin{equation}\label{eq: decay Laurent Left}
            \frac{|\vect{u}^{(i + k)}|}{|\vect{u}^{(i)}|} = e^{-(r-\tilde{\beta})}.
        \end{equation}
        \item Past the defect site, that is for $i \geq m$,
        \begin{equation}\label{eq: decay Laurent Right}
            \frac{|\vect{u}^{(i + k)}|}{|\vect{u}^{(i)}|} = e^{-(r+\tilde{\beta})}.
        \end{equation}
    \end{itemize}
    In both cases $r=\frac{1}{2}\log\Bigl( \prod_{i =1}^k \frac{b_i}{c_i} \Bigr)$ is fixed and $\tilde{\beta}(\lambda) = \arccosh{-\frac{g(\lambda)}{2Ae^{r}}}$ depends on the eigenvalue $\lambda$ and where we choose the positive root of the $\arccosh{\cdot}$.
\end{corollary}
\begin{proof}
    For an eigenvalue outside the Floquet spectrum, it follows by Theorem \ref{Thm: alpha and beta fixed} that $\alpha = \{0, \pi \}$ is fixed. 
    So, if $\lambda$ is an eigenvalue of $f(e^{-\i(\alpha + \i( r + \tilde{\beta}))})\in\R^{k\times k}$ with the eigenvector $\vect{v}\in \R^{k\times 1}$ then $\lambda$ is also an eigenvalue of $f(e^{-\i(\alpha + \i (r - \tilde{\beta}))})\in\R^{k\times k}$, with eigenvector $\vect{w}\in \R^{k\times 1}$. Suppose that $\vect{v}$ is linearly independent of $\vect{w}$, if not, a linearly independent eigenvector may be recovered by the same techniques as present in \cite{ammari2024spectra}. It is not hard to see that the quasiperiodic extension
    \begin{equation}
        \vect{u}_1 = \bigl(\dots, e^{-2\i(\alpha + \i( r - \tilde{\beta}))}\vect{w}^\top, e^{-\i(\alpha + \i( r - \tilde{\beta}))} \vect{w}^\top,  \vect{w}^\top, e^{\i(\alpha + \i( r - \tilde{\beta}))}\vect{w}^\top, e^{2\i(\alpha + \i( r - \tilde{\beta}))} \vect{w}^\top, \dots \bigr)
    \end{equation}
    is an eigenvector of $\vect{L}(f)$ everywhere but in the defected row. In the same way,
    \begin{equation}
        \vect{u}_2 = \bigl(\dots, e^{-2\i(\alpha + \i( r + \tilde{\beta}))} \vect{v}^\top, e^{-\i(\alpha + \i( r + \tilde{\beta}))} \vect{v}^\top, \vect{v}^\top, e^{\i(\alpha + \i( r + \tilde{\beta}))} \vect{v}^\top, e^{2\i(\alpha + \i( r + \tilde{\beta}))} \vect{v}^\top, \dots\bigr)
    \end{equation}
    is an eigenvector for $\vect{L}(f)$ everywhere except in the defected row. As a consequence a linear combination of the left hand side of $\vect{u}_1$ and the right hand side of $\vect{u}_2$ yields the vector
    \begin{equation}\label{eq: eigenvector defected laurent operator}
        \vect{u} = \bigl(\dots, a e^{2\i(\alpha + \i( r - \tilde{\beta}))}\vect{w}^\top, a e^{\i(\alpha + \i( r - \tilde{\beta}))} \vect{w}^\top,  a \vect{w}^\top + b\vect{v}^\top, be^{\i(\alpha + \i( r + \tilde{\beta}))} \vect{v}^\top, be^{2\i(\alpha + \i( r + \tilde{\beta}))} \vect{v}^\top, \dots\bigr)
    \end{equation}
    which is an exact eigenvector of the Laurent operator $\vect{L}(f)$. The desired decay bounds presented in \eqref{eq: decay Laurent Left} and \eqref{eq: decay Laurent Right} are achieved by the construction of $\vect{u}$. The explicit value of $\tilde{\beta}$ is achieved by solving \eqref{eq: eigval of symbol} for $\tilde{\beta}$ and therefore the statement follows.
\end{proof}

It is worth discussing the relation between the complex band structure and the winding region of the symbol. For any eigenvalue $\lambda \in \sigma_{\mathrm{wind}}$, the associated eigenvectors of the Toeplitz operator are exponentially localised on the edge. In terms of the complex band structure, the winding region corresponds to quasimomenta with a strictly positive imaginary part, that is, $r + \tilde{\beta} > 0$ and $r - \tilde{\beta} > 0$.
The advantage of the complex band structure over the non-trivial winding region is its ability to provide an explicit value for complex quasiperiodicity, offering sharp quantitative decay estimates compared to the qualitative assessment of the region of non-trivial winding.

For an eigenvalue $\lambda\in\sigma_{\mathrm{wind}}^\mathsf{c}$ the associated eigenvectors are no longer edge localised, and we may instead obtain localisation within the bulk by introducing a compact defect. Crucially, the complex band structure may still be used to characterise the decay properties of such eigenvectors, where the decay on each side of the defect is predicted by corresponding branch of the complex band structure. For $\lambda \in \sigma_{\mathrm{det}}$, it holds that $r-\tilde{\beta} =  0$ and there may be eigenvectors without decay to one side of the defect.

Figure \ref{fig: Spectral plot zones} provides a schematic illustration of defect eigenmodes in different spectral regions. These were computed using the defective capacitance matrix, which is an example of  tridiagonal $k$-Toeplitz matrix which will be discussed in detail in  Section \ref{sec:defect_inf}.

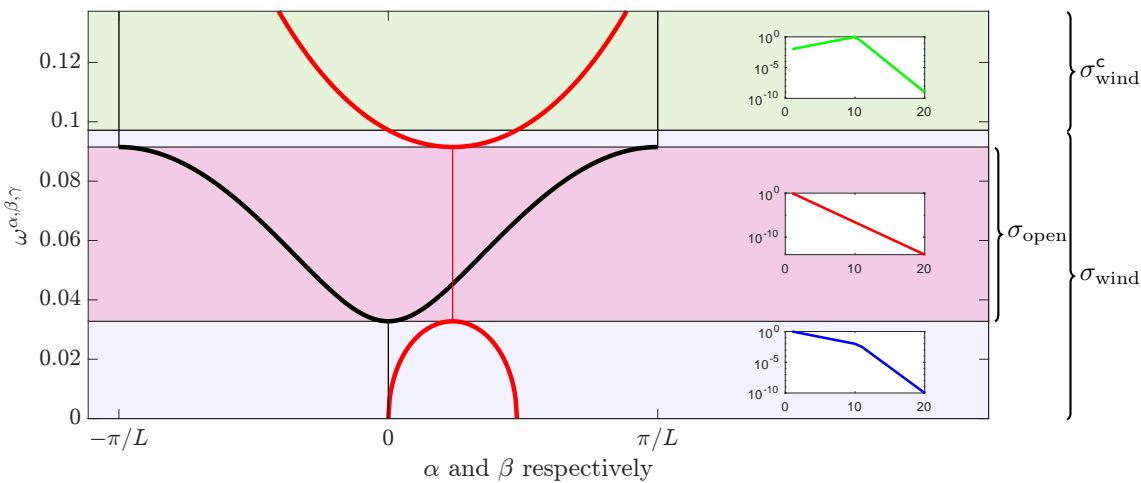
\begin{figure}[tbh]
    \centering
    \begin{tikzpicture}
    \draw (0,0) node{\includegraphics[width=0.80\linewidth]{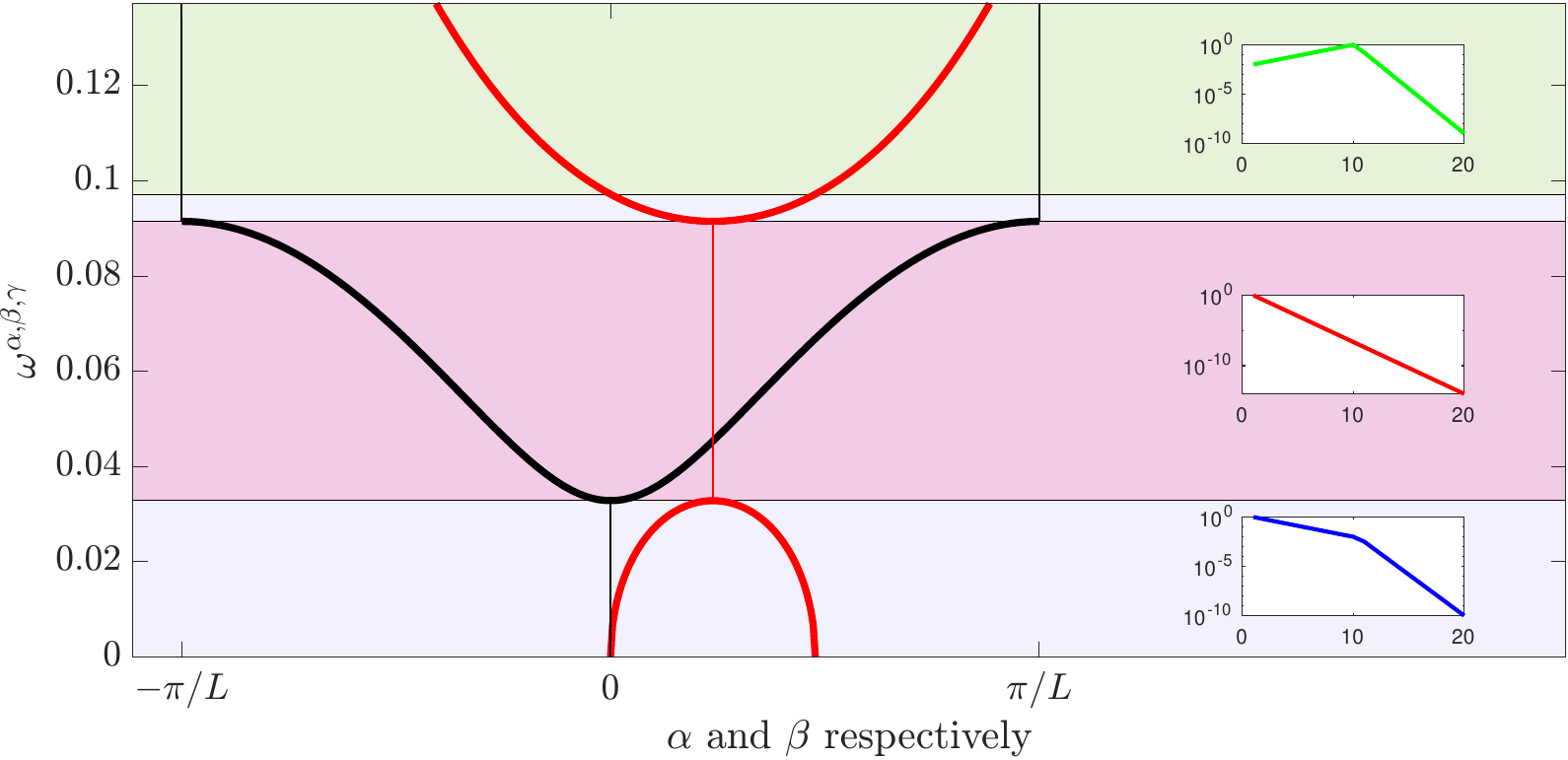}};
    \draw[thick,decorate,
	decoration =brace] (6.55,1.3) -- (6.55,-1) node[pos=0.5,right,xshift=1pt]{{$\sigma_{\mathrm{open}}$}};
    \draw[thick,decorate,
	decoration =brace] (7.5,1.5) -- (7.5,-2.3) node[pos=0.5,right,xshift=1pt]{$\sigma_{\mathrm{wind}}$};
    \draw[thick,decorate,
	decoration =brace] (7.5,3.1) -- (7.5,1.52) node[pos=0.5,right,xshift=1pt]{$\sigma_{\mathrm{wind}}^\mathsf{c}$};
    \end{tikzpicture}
    \caption{The light blue area, labelled $\sigma_{\mathrm{wind}}$, depicts the positive winding zone where eigenvectors are exponentially localized on one side of the system, although they may exhibit varying decay rates. Within the fully encompassed red area, called the open limit spectrum $\sigma_{\mathrm{open}}$, eigenmodes are uniformly localised on one side with consistent decay. The green region, $ \sigma_{\mathrm{wind}}^\mathsf{c}$, signifies zones where defect eigenmodes are exponentially localized within the bulk, displaying differing localisation strengths on either side of the defect. The calculations are performed based on the Capacitance Toeplitz matrix detailed in Section \ref{sec: quasiperiodic capacitance matrix} 
    }
    \label{fig: Spectral plot zones}
\end{figure}

\subsection{Eigenvectors of tridiagonal $k$-Toeplitz and $k$-Laurent matrices}\label{Sec: Pseudospectrum}

Despite the inapplicability of Bloch's theorem to finite resonator chains due to translational symmetry loss, this section demonstrates that the complex band structure can still accurately characterise localisation in finite systems, with an exponentially small error.

As pointed out in \cite{REICHEL1992153}, the spectra of non-Hermitian systems are highly sensitive to small perturbations. The solution is offered by forming $\varepsilon$-pseudoeigenvectors through the truncation of localised Bloch modes. This leads us to introduce pseudospectra. For a detailed discussion, see \cite[Section 1-2]{trefethen.embree2005Spectra}.

\begin{definition}\label{def: pseudospectrum}
    Let \(\varepsilon > 0\). Then, \(\lambda \in \mathbb{C}\) is an \(\varepsilon\)-pseudoeigenvalue of \(\mathbf{A}\in \mathbb{C}^{N\times N}\) if one of the following conditions is satisfied:
    \begin{enumerate}[(i)]
        \item \(\lambda\) is a proper eigenvalue of \(\mathbf{A} + \mathbf{E}\) for some \(\mathbf{E} \in \mathbb{C}^{N\times N}\) such that \(\lVert \mathbf{E} \rVert \leq \varepsilon\);
        \item $\lVert (\mathbf{A}-\lambda \Id) \mathbf{u} \rVert < \varepsilon$ for some vector $u$ with $\lVert \vect u \rVert = 1$;
        \item $\lVert (\mathbf{A}- \lambda \Id)^{-1} \rVert^{-1} \leq \epsilon$.
    \end{enumerate}
    The set of all $\varepsilon$-pseudoeigenvalues of $\vect A$, the $\varepsilon$-pseudospectrum, is denoted by $\sigma_{\epsilon}(\vect A)$. If some non-zero $\vect u$ satisfies $\lVert (\lambda\Id - \mathbf{A}) \mathbf{u} \rVert < \varepsilon$, then we say that $\vect u$ is an $\varepsilon$-pseudoeigenvector of $\mathbf{A}$. By the norm equivalence in finite dimensions any norm may be used.
\end{definition}
There are instances where the pseudospectrum aligns with the spectrum. This is formally asserted by the subsequent theorem \cite[Theorem 2.3]{trefethen.embree2005Spectra,}.

\begin{theorem}[Bauer-Fike] \label{Thm: Bauer_Fike Theorem}
    Suppose $\vect{A}  \in \C^{N \times N}$ is diagonalisable, $\vect{A} = \vect{V} \vect{\Lambda} \vect{V}^{-1}$. Then, for each $\varepsilon > 0$, 
    \begin{equation}
        \sigma(\vect{A} ) + \Delta_\varepsilon \subseteq \sigma_\varepsilon(\vect{A} ) \subseteq \sigma(\vect{A} ) + \Delta_{\varepsilon \kappa(\vect{V})},
    \end{equation}
    where $\kappa(\vect{V})$ is the condition number of $\vect{V}$ and 
    \begin{equation}
        \sigma(\vect{A} ) + \Delta_\varepsilon = \bigl\{ z: z = z_1 + z_2, ~z_1\in \sigma(\vect{A} ),~ |z_2| < \varepsilon \bigr\}.
    \end{equation}
\end{theorem}

In the sequel, we will consider defected tridiagonal operators. We assume that the defect position remains the same relative to the truncated structure, that is, $m = \lfloor N/2\rfloor$ for a defect located in the middle of the resonator chain.

For defective non-Hermitian operators, we derive a result pertaining to the convergence of defect eigenfrequencies.
Let us denote by $\vect{T}_N(f) \in \R^{N \times N}$ the tridiagonal $k$-Toeplitz matrix which results from truncating the operator $\vect{T}(f)$ to the first $N$ entries and let $\lambda_N$ be the associated eigenvalue. 

\begin{corollary}
    Let $\lambda \in \R \setminus \sigma_\text{open}$ be the defect eigenfrequency of the Toeplitz operator $\vect{T}(f)$ and $\lambda_N$ the defect eigenfrequency of the truncated operator $\vect{T}_N(f)$, then
    \begin{equation}
        |\lambda - \lambda_N| = \mathcal{O}(e^{-\tilde{\beta}N}) ~~\text{ as }N\to \infty,
    \end{equation}
    where $\tilde{\beta}(\lambda) = \arccosh{-\frac{g(\lambda)}{2Ae^{r}}}$ where the positive root of $\arccosh{\cdot}$ is chosen.
\end{corollary}

\begin{proof}
    
Let us introduce a similarity transform which scales the argument of the symbol function, $\vect{T}_N\bigl(f(e^r\cdot)\bigr) = \vect{R}_N(e^r)^{-1}\vect{T}_N\vect{R}_N(e^r)$, where  $\bigl(\vect{R}_N(e^r)\bigr)_{ii} := e^{-r\lfloor\frac{i-1}{k}\rfloor}$.
In order to apply Theorem \ref{Thm: Bauer_Fike Theorem} we have to control the condition number $\kappa(\vect{V})$ for the eigenbasis of $\vect{T}_N\bigl(f(e^r\cdot)\bigr) = \vect{V}\vect{\Lambda}\vect{V}^{-1}$.
From the proof of \cite[Theorem 3.5]{ammari2024spectra}, we have the bound  $\kappa(\vect{V}) \leq C$ which is uniform in $N$. So, by Theorem \ref{Thm: Bauer_Fike Theorem}  it follows that,
\begin{equation}\label{eq: pseudospectrum bound collapsed toeplitz operator}
\sigma_{\varepsilon}\Bigl(\vect{T}_N\bigl(f(e^r\cdot)\bigr)\Bigr) \subset \sigma\Bigl(\vect{T}_N\bigl(f(e^r\cdot)\bigr)\Bigr) + \Delta_{C\varepsilon},
\end{equation}
A bound on the pseudospectrum can be achieved by the fact that defect modes are exponentially localised around the deffect site. By \eqref{eq: pseudospectrum bound collapsed toeplitz operator} we achieve the bound on the defect eigenfrequency,
\begin{equation}\label{eq: bound on collapsed defect eigenfrequency}
    \left|\lambda^{\vect{T}\bigl(f(e^s\cdot)\bigr)} - \lambda^{\vect{T}_N\bigl(f(e^s\cdot)\bigr)}_N\right| \leq C e^{-\tilde{\beta}N}.
\end{equation}
Since $\vect{T}_N\bigl(f(e^r\cdot)\bigr)$ is similar to $\vect{T}_N(f)$ and therefore share the same eigenvalues, it readily follows from \eqref{eq: bound on collapsed defect eigenfrequency} that,
\begin{equation}
    \left|\lambda^{\vect{T}(f)} - \lambda^{\vect{T}_N(f)}_N \right| \leq C e^{-\tilde{\beta}N},
\end{equation}
which concludes the proof.
\end{proof}

\begin{theorem}\label{thm: defect eigenfrequncy pseudoeigenvector wind}
    Let $\lambda \in \sigma_\text{wind}\setminus \sigma_\text{det}$ and let $\vect{v}_{0,N}$ be a truncated eigenvector of the defected Toeplitz operator $\vect{T}(f)$, then $\vect{v}_{0,N}$ is an $\varepsilon_N$-pseudoeigenvector of $\vect{T}_N(f)$, with $\varepsilon_N = \mathcal{O}(e^{-B \lfloor N/2k \rfloor})$, where $B =  r-\tilde{\beta}(\lambda)>0$ for $r = \frac{1}{2}\log\bigl( \prod_{i =1}^k \frac{b_i}{c_i} \bigr)$ and $\tilde{\beta}(\lambda) = \arccosh{-\frac{g(\lambda)}{2Ae^{r}}}$ where the positive root of  $\arccosh{\cdot}$ is chosen.
\end{theorem}

\begin{proof}
    The truncated eigenvector $\vect{v}_{0,N}$ satisfies the eigenvalue problem $\vect{T}_N\vect{v}_{0,N} = \lambda_N\vect{v}_{0,N}$ in the rows leading up to the defect situated at $m = \lfloor N/2 \rfloor$. This allows for the following estimate,
    \begin{equation}
        \Bigl\rVert \Bigl(\vect{T}_{N} - \lambda\Id_N\Bigr) \vect{v}_{0,N}\Bigr\rVert_2 = \left \lVert \left(\begin{smallmatrix}
            \vect{0}_{m-1}\\
            b_k\vect{v}_{0,N}^{(m)}\\
            \vdots\\
            \vect{v}_{0,N}^{(m)}
        \end{smallmatrix}\right)\right\rVert_2 \leq K \sum_{i = m}^N \bigl(\vect{v}_{0,N}^{(i)} \bigr)^2 =: \varepsilon_N.
    \end{equation}
    By the estimate on the $N$-th entry of the eigenvector in \eqref{eq:decay_outside}, it holds
    \begin{equation}
        \sum_{i = m}^N \bigl( \vect{v}_{0,N}^{(i)} \bigr)^2 \leq \vect{v}_{0,N}^{(m)} \leq \mathcal{O}(e^{-(r-\tilde{\beta})\lfloor N/2k\rfloor}),
    \end{equation}
    which achieves the desired bound on $\varepsilon_N$. Since $\lambda \in \sigma_\text{wind}\setminus\sigma_\text{det}$ it holds that $r-\tilde{\beta}(\lambda) > 0$.
\end{proof}

Let $\vect{L}(f)$ be a tridiagonal Laurent operator with a defect a site $m$ and let $\vect{L}_N(f)$ be the truncated version, where the truncation is centred around the defect site. A pseudoeigenvector for a Laurent matrix is constructed as follows.

\begin{theorem}\label{thm: defect eigenfrequncy pseudoeigenvector}
    Let $\lambda \in \sigma_{\mathrm{wind}}^\mathsf{c}\setminus \sigma_\text{det}$ and let $\vect{v}_{0,N}$ be a truncated eigenvector of the defected Laurent operator $\vect{L}(f)$, then $\vect{v}_{0,N}$ is an $\varepsilon_N$-pseudoeigenvector of $\vect{L}_N(f)$, with $\varepsilon_N = \mathcal{O}(e^{-B \lfloor N/2k \rfloor})$, where $B:=\tilde{\beta}(\lambda)-r > 0$, for $r=\frac{1}{2}\log\Bigl( \prod_{i =1}^k \frac{b_i}{c_i} \Bigr)$ and $\tilde{\beta}(\lambda) = \arccosh{-\frac{g(\lambda)}{2Ae^{r}}}$  where the positive root of  $\arccosh{\cdot}$ is chosen.
    
\end{theorem}

\begin{proof}
    The truncated eigenvector $\vect{v}_{0,N}$ satisfies the eigenvalue problem $\vect{L}_N\vect{v}_{0,N} = \lambda_N\vect{v}_{0,N}$ in all but the first and last rows. This allows for the following estimate,
    \begin{equation}\label{eq: eigenvector estimate}
        \Bigl\rVert \Bigl(\vect{L}_N - \lambda_N\Id_N\Bigr) \vect{v}_{0,N}\Bigr\rVert_2 = \left \lVert \left(\begin{smallmatrix}
            -c_k \vect{v}_{0,N}^{(1)}\\
            \vect{0}_{N-2}\\
            b_k\vect{v}_{0,N}^{(N)}
        \end{smallmatrix}\right)\right\rVert_2 =  c_k^2\bigl(\vect{v}_{0,N}^{(1)} \bigr)^2  +b_k^2\bigl(\vect{v}_{0,N}^{(N)} \bigr)^2 =: \varepsilon_N, 
    \end{equation}
    where $c_k$ and $b_k$ are the entries of the $k$-Toeplitz matrix which are constant and independent of $N$. Bounds on the first and last eigenvector entries can be achieved as follows. Since the defect is situated at the index $m = \lfloor N/2 \rfloor$, by Corollary \ref{cor: decay rate outside spectrum} the first and last entries are therefore explicit, so
    \begin{equation}
        \vect{v}_{0,N}^{(N)} = K e^{-(r+\tilde{\beta})\lfloor N/2k \rfloor} \quad\text{and}\quad \vect{v}_{0,N}^{(1)} = K e^{(r-\tilde{\beta})\lfloor N/2k \rfloor}.
    \end{equation}
    Let us define $B:= \min( r + \tilde{\beta}, \tilde{\beta} - r) = \tilde{\beta} -r$, then as $\lambda \in \sigma_{\mathrm{wind}}^\mathsf{c}\setminus \sigma_\text{det}$ it follows that $B > 0$. By \eqref{eq: eigenvector estimate} it holds that 
    \begin{equation}
        \varepsilon_N \leq \mathcal{O}(e^{-B\lfloor N/2k\rfloor})\quad \text{as }N \to \infty.
    \end{equation}
    This shows that $\lambda_N$ is an $\varepsilon_N$-pseudoeigenvalue of $\vect{L}_{N}(f)$ and concludes the proof. 
\end{proof}

In the case $\lambda \in \sigma_\text{det}$ the construction for exponentially good pseudoeigenvectors as in Theorems \ref{thm: defect eigenfrequncy pseudoeigenvector wind} and \ref{thm: defect eigenfrequncy pseudoeigenvector} fails, through the correct normalisation one can however achieve an algebraic bound on $\varepsilon_N$.
Figure \ref{fig: Spectral Convergence} demonstrates that Theorems \ref{thm: defect eigenfrequncy pseudoeigenvector wind} and \ref{thm: defect eigenfrequncy pseudoeigenvector} offer correct and sharp exponential bounds for the pseudospectrum.

\begin{figure}[hbt]
    \centering
    \subfloat[][Convergence of defect eigenfrequency.]
    {\includegraphics[height=0.18\linewidth]{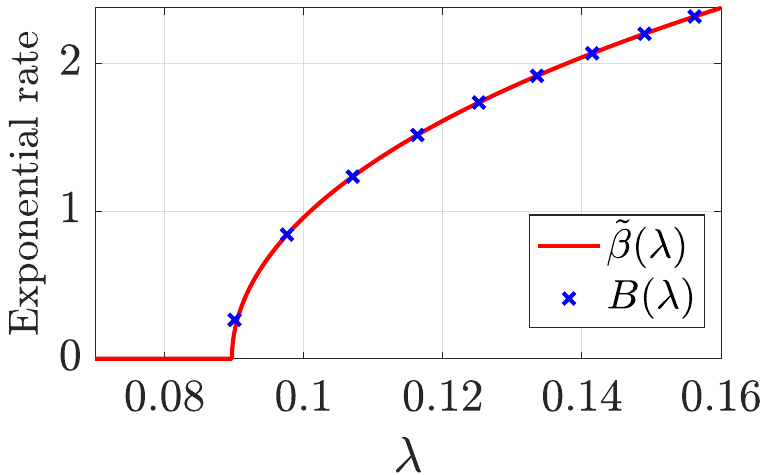}}\quad
    \subfloat[][Convergence of the Pseudospectrum for the Toeplitz operator.]
    {\includegraphics[height=0.18\linewidth]{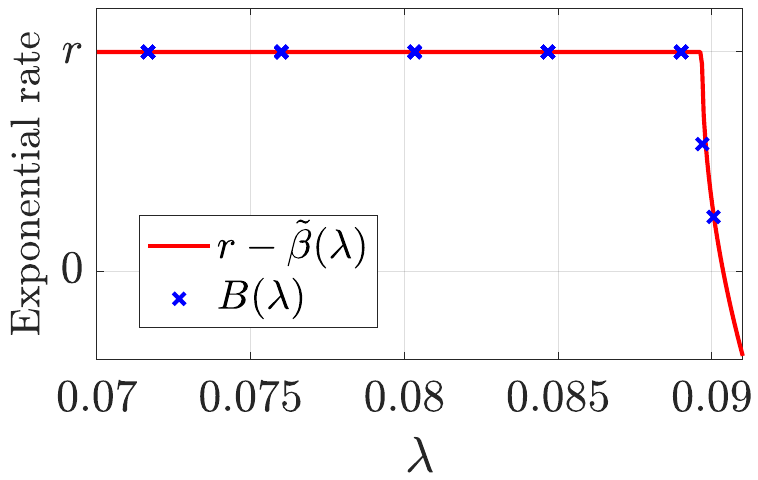}}\qquad
    \subfloat[][Convergence of the Pseudospectrum for the Laurent operator.]
    {\includegraphics[height=0.18\linewidth]{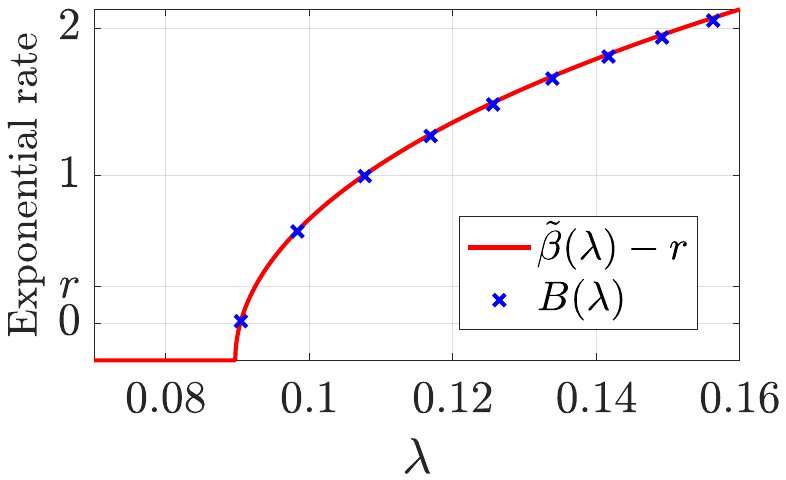}}
    \caption{ Truncated eigenvectors produce $\varepsilon_N$-pseudoeigenvectors. We examine the analytical convergence bounds of these pseudoeigenvectors, represented by the solid red line, against numerical convergence rates across various band gap frequencies of $\lVert (\vect{A}_N-\lambda_N)\vect{v}_{0,N}\rVert \leq e^{-B(\lambda)N}=: \varepsilon_N$  marked by the blue crosses. The close agreement observed signifies that the presented bounds are sharp estimates.}
    \label{fig: Spectral Convergence}
\end{figure}

The construction of $\varepsilon$-pseudoeigenvectors by truncating associated eigenvectors has the advantage of maintaining the same decay properties.
Therefore, the complex band structure provides decay estimates for eigenvectors of finite tridiagonal $k$-Toeplitz matrices.

\section{Subwavelength Localisation in Defected non-Hermitian Resonator Chains}\label{Sec: quasiperiodic Gauge capacitance}
Next, we will illustrate the main results of Section \ref{sec: tridiagonal k toeplitz} in the setting of one-dimensional subwavelength wave physics. The results detailed in \cite{FoundationsSkinEffect} have shown that the \emph{gauge capacitance matrix} yields an asymptotic description of the resonances in a chain of subwavelength resonators. In particular, the capacitance matrix has the form of a non-Hermitian tridiagonal $k$-Toeplitz matrix.

\subsection{Setting and Problem Formulation}

 The introduction of an imaginary gauge potential $\gamma$ breaks the hermiticity of a subwavelength resonator chain. We consider a unit cell $Y$ of length $L$ and spanning the lattice $\Lambda$.  A unit cell $Y$ contains $k$  resonators \(D := (x_i^{\mathrm{L}}, x_i^{\mathrm{R}}\)), where \((x_i^{\mathrm{L}, \mathrm R})_{1 \leq i \leq k} \subset \R\) are the $2k$ extremities such that $x_i^{\mathrm{L}} < x_i^{\mathrm{R}} < x_{i+1}^{\mathrm{L}}$ for any $1 \leq i \leq k$. The lengths of the $i$-th resonators will be denoted by $\ell_i = x_i^{\mathrm{R}}-x_i^{\mathrm{L}}$ and the spacings between the $i$-th and the $(i+1)$-th resonators will be denoted by $s_i = x_{i+1}^{\mathrm{L}}-x_i^{\mathrm{R}}$. Furthermore, we assume that the resonators are $k$ periodically repeated, that is, $s_{i+k} = s_i$. This setup is illustrated in Figure \ref{fig:setting}.

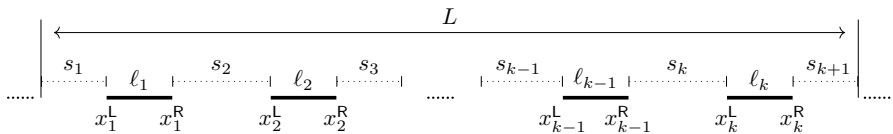
\begin{figure}[htb]
    \centering
    \begin{adjustbox}{width= 12cm} 
    \begin{tikzpicture}
        \coordinate (x1l) at (1,0);
        \path (x1l) +(1,0) coordinate (x1r);
        \coordinate (s0) at (0.5,0.7);
        \path (x1r) +(0.75,0.7) coordinate (s1);
        
        \path (x1r) +(1.5,0) coordinate (x2l);
        \path (x2l) +(1,0) coordinate (x2r);
        \path (x2r) +(0.5,0.7) coordinate (s2);
        \path (x2r) +(1,0) coordinate (x3l);
        \path (x3l) +(1,0) coordinate (x3r);
        \path (x3r) +(1,0.7) coordinate (s3);
        \path (x3r) +(2,0) coordinate (x4l);
        \path (x4l) +(1,0) coordinate (x4r);
        \path (x4r) +(0.4,0.7) coordinate (s4);
        \path (x4r) +(1,0) coordinate (dots);
        \path (dots) +(1,0) coordinate (x5l);
        \path (x5l) +(1,0) coordinate (x5r);
        \path (x2r) +(1.2,0) coordinate (x6l);
        \path (x5r) +(0.875,0.7) coordinate (s5);
        \path (x6l) +(1,0) coordinate (x6r);
        \path (x6r) +(1.25,0) coordinate (x7l);
        \path (x6r) +(0.57,0.7) coordinate (s6);
        \path (x7l) +(1,0) coordinate (x7r);
        \path (x7r) +(1.5,0) coordinate (x8l);
        \path (x7r) +(0.75,0.7) coordinate (s7);
        \path (x8l) +(1,0) coordinate (x8r);
        \coordinate (s8) at (12,0.7);

        \draw[dotted, line cap=round, line width=1pt, dash pattern=on 0pt off 2\pgflinewidth] ($(x1l) - (1.5, 0)$) -- ($(x1l)- (1.1, 0)$);
        \draw ($(x1l)- (1, 0)$) -- ++(0,1.2);

        \draw[dotted,|-|] ($(x1l)- (1, -0.25)$) -- ($(x1l) + (0, 0.25)$);
        
        \draw[ultra thick] (x1l) -- (x1r);
        \node[anchor=north] (label1) at (x1l) {$x_1^{\iL}$};
        \node[anchor=north] (label1) at (x1r) {$x_1^{\iR}$};
        \node[anchor=south] (label1) at ($(x1l)!0.5!(x1r)$) {$\ell_1$};
        \draw[dotted,|-|] ($(x1r)+(0,0.25)$) -- ($(x2l)+(0,0.25)$);
        \draw[ultra thick] (x2l) -- (x2r);
        \node[anchor=north] (label1) at (x2l) {$x_2^{\iL}$};
        \node[anchor=north] (label1) at (x2r) {$x_2^{\iR}$};
        \node[anchor=south] (label1) at ($(x2l)!0.5!(x2r)$) {$\ell_2$};
        \draw[dotted,|-|] ($(x2r)+(0,0.25)$) -- ($(x3l)+(0,0.25)$);
        \draw[dotted, line cap=round, line width=1pt, dash pattern=on 0pt off 2\pgflinewidth] ($(x3l)+(0.4,0)$) -- ($(x6r)-(0.4,0)$);
        
        \draw[dotted,|-|] ($(x6r)+(0,0.25)$) -- ($(x7l)+(0,0.25)$);
        \draw[ultra thick] (x7l) -- (x7r);
        \node[anchor=north] (label1) at (x7l) {$x_{k-1}^{\iL}$};
        \node[anchor=north] (label1) at (x7r) {$x_{k-1}^{\iR}$};
        \node[anchor=south] (label1) at ($(x7l)!0.5!(x7r)$) {$\ell_{k-1}$};
        \draw[dotted,|-|] ($(x7r)+(0,0.25)$) -- ($(x8l)+(0,0.25)$);
        \draw[ultra thick] (x8l) -- (x8r);
        \draw[dotted,|-|] ($(x8r)+(0,0.25)$) -- ($(x8r)+(1,0.25)$);
        \draw ($(x8r)+ (1,0)$) -- ++(0,1.2);  

        \draw[<->] ($(x1l) + (-0.8,1)$) -- ($(x8r) + (0.8,1)$) node[midway, above] {$L$};
    
        \draw[dotted, line cap=round, line width=1pt, dash pattern=on 0pt off 2\pgflinewidth] ($(x8r)+(1.1, 0)$) -- ($(x8r) + (1.5, 0)$);
        \node[anchor=north] (label1) at (x8l) {$x_{k}^{\iL}$};
        \node[anchor=north] (label1) at (x8r) {$x_{k}^{\iR}$};
        \node[anchor=south] (label1) at ($(x8l)!0.5!(x8r)$) {$\ell_k$};
        \node[anchor=north] (label1) at (s0) {$s_1$};
        \node[anchor=north] (label1) at (s1) {$s_2$};
        \node[anchor=north] (label1) at (s2) {$s_3$};
        
        \node[anchor=north] (label1) at (s6) {$s_{k-1}$};
        \node[anchor=north] (label1) at (s7) {$s_{k}$};
        \node[anchor=north] (label1) at (s8) {$s_{k+1}$};
    \end{tikzpicture}
    \end{adjustbox}
    \caption{Unit cell $Y$ of an infinite chain of subwavelength resonators, with lengths $(\ell_i)_{1\leq i\leq k}$ spacings $(s_{i})_{1\leq i\leq k+1}$ spanned by the lattice vector $\Lambda$ of length $L$.}
    \label{fig:setting}
\end{figure}

The Helmholtz equation with imaginary gauge potential of amplitude $\gamma \in \R\setminus \{0\}$ is given by 
\begin{equation}\label{eq: wave equation skin effect}
\begin{cases}u^{\prime \prime}(x)+\gamma u^{\prime}(x)+\frac{\omega^2}{v_i^2} u(x) =0, & x \in D_i \\ u^{\prime \prime}(x)+\frac{\omega^2}{v_0^2}u(x)=0, & x \in \mathbb{R} \backslash D, \\ \left.u\right|_{\iR}(x_i^{\iL, \iR})-\left.u\right|_{\iL}(x_i^{\iL, \iR})=0, & \text {for all } 1 \leq i \leq k, \\ \left.\frac{\mathrm{d} u}{\mathrm{~d} x}\right|_{\iR}\left(x_i^{\iL}\right)=\left.\delta \frac{\mathrm{d} u}{\mathrm{~d} x}\right|_{\iL}\left(x_i^{\iL}\right), & \text {for all } 1 \leq i \leq k, \\  \left.\frac{\mathrm{d} u}{\mathrm{~d} x}\right|_{\iL}\left(x_i^{\iR}\right)=\left. \delta \frac{\mathrm{d} u}{\mathrm{~d} x}\right|_{\iR}\left(x_i^{\iR}\right), & \text {for all } 1 \leq i \leq k, \\ 
u(x + \ell) = e^{\i(\alpha + \i \beta)\ell}u(x), &\text{for all } \ell \in \Lambda.
\end{cases}
\end{equation}
In particular, we follow Theorem \ref{Thm: spectral limit CBZ} and consider a complex Floquet-Bloch condition in the last line of \eqref{eq: wave equation skin effect}.

\subsection{Subwavelength Regime}\label{sec: quasiperiodic capacitance matrix}
In the subwavelength limit, one considers a medium with \emph{high contrast inclusions} defined by $ 0 < \delta \ll 1$ (see, for example, \cite{ammari.fitzpatrick.ea2018Mathematical,ammari2024functional}).
\begin{definition}
    Given $\delta > 0$, a \emph{subwavelength resonant frequency} $\omega = \omega(\delta)$ is defined to have non-negative real part and,
    \begin{enumerate}[(i)]
    \item there exists a non-trivial solution to the scattering problem \eqref{eq: wave equation skin effect}.
    \item $\omega$ depends continuously on $\delta$ and satisfies $\omega(\delta) \xrightarrow{\delta \to 0} 0$.
    \end{enumerate}
\end{definition}

\begin{definition}Consider solutions $V^{\alpha, \beta}_i$ $:\mathbb{R} \to \R$ of the problem
\begin{equation}\label{eq: chain equation}
    \begin{cases}
        \frac{\mathrm{d}^2}{\mathrm{d}x^2}V^{\alpha, \beta}_i, = 0, & x\in Y \setminus D,\\
        V^{\alpha, \beta}_i(x) = \delta_{ij}, & x \in D_j, \\
        V^{\alpha, \beta}_i(x + mL) = e^{\i (\alpha + \i\beta)mL}V^{\alpha, \beta}_i(x), & m \in \Z,
    \end{cases}
\end{equation}
for $1 \leq i, j \leq k$. Then the \emph{quasiperiodic capacitance matrix} is defined coefficient-wise by
     \begin{equation}\label{eq: cap}
         \widehat{C}_{ij}^{\alpha, \beta} = - \frac{l_i}{\int_{D_i} e^{\gamma x}\mathrm{d}x} \left(-e^{\gamma x_j^\iL}\left.\frac{\mathrm{d} V^{\alpha, \beta}_i}{\mathrm{d}x}\right|_\iL(x^\iL_j) + e^{\gamma x_j^\iR} \left. \frac{\mathrm{d} V^{\alpha, \beta}_i}{\mathrm{d}x}\right|_\iR(x_j^\iR) \right).
     \end{equation}
     We also define the \emph{generalised quasiperiodic capacitance matrix}, denoted by $\widehat{\mathcal{C}}^{\alpha, \beta}=(\widehat{\mathcal{C}}^{\alpha, \beta}_{ij})\in\mathbb{C}^{k\times k}$, to be the square matrix given by
		\begin{equation*}
			\widehat{\mathcal{C}}^{\alpha, \beta}_{ij}=\frac{\delta v_i^2}{l_i} \widehat{C}^{\alpha, \beta}_{ij}, \quad i,j=1,\dots,k.
		\end{equation*}
	\end{definition}

\begin{lemma}
    For $\gamma \in \R\setminus\{0\}$, the quasiperiodic capacitance matrix is given by
    \begin{align}\label{eq: quasiperiodic capacitance formula}
        \widehat{C}^{\alpha, \beta}_{ij} &= \left( \frac{\gamma}{s_i}\frac{\ell_i}{1-e^{-\gamma \ell_i}} - \frac{\gamma}{s_{i-1}} \frac{\ell_i}{1-e^{\gamma \ell_i}}\right) \delta_{ij} + \left(-\frac{\gamma}{s_i}\frac{\ell_j}{1-e^{-\gamma\ell_i}}\right)\delta_{i(j-1)} + \left(\frac{\gamma}{s_j}\frac{\ell_j}{1-e^{\gamma \ell_j}}\right)\delta_{i(j+1)} \nonumber\\
        &\quad + \left(-e^{-\i(\alpha + \i \beta)L}\frac{\gamma}{s_N}\frac{\ell_1}{1-e^{-\gamma\ell_1}}\right)\delta_{ik}\delta_{j1} + \left(e^{\i(\alpha + \i\beta)L}\frac{\gamma}{s_N}\frac{\ell_N}{1-e^{\gamma \ell_N}}\right)\delta_{i1}\delta_{jk}.
    \end{align}
\end{lemma}

\begin{proof}
     The solutions to $(\ref{eq: chain equation})$ for $2 \leq i \leq k-1$ are given by
    \begin{equation}
        V^{\alpha, \beta}_i(x) = \begin{cases}
            \frac{1}{s_{i-1}}(x-x_i^\iL), & x_{i-1}^\iR \leq x \leq x^\iL_i, \\
            1, & x_i^\iL \leq x \leq x_i^\iR,\\
            -\frac{1}{s_i}(x-x^\iL_{i+1}), & x_i^\iR \leq x \leq x_{i+1}^\iL, \\
            0, & \text{else},
         \end{cases}
    \end{equation}
     while for the bigger and smaller $i$ we multiply by the corresponding $e^{\i (\alpha + \i\beta)mL}$-factor. 
     The assertion now follows from a direct evaluation of \eqref{eq: cap}.
\end{proof}
The capacitance matrix captures the behaviour of the spectral bands in the subwavelength limit $\delta \to 0$.
\begin{proposition}\cite[Proposition 2.5.]{debruijn2024complexbandstructuresubwavelength}\label{Prop:  Gauge Capacitance Thm}
    The $k$ subwavelength (complex) band functions $\alpha + \i \beta \mapsto \omega^{\alpha, \beta}$ satisfy as $\delta \to 0$,
    \begin{equation}\label{eq:omega}
        \omega_i^{\alpha,  \beta} = \sqrt{ \lambda_i^{\alpha , \beta}} + \mathcal{O}(\delta),
    \end{equation}
    where $(\lambda_i^{\alpha, \beta})_{1 \leq i \leq k}$ are the eigenvalues of the eigenvalue problem
    \begin{equation}
        \widehat{\mathcal{C}}^{\alpha, \beta} \vect{a}_i = \lambda_i^{\alpha, \beta}\vect{a}_i, \quad 1 \leq i \leq k.
    \end{equation}
    We select the $k$ values of $\pm\sqrt{\lambda_i^{\alpha, \beta}}$ that have non-negative real parts.
\end{proposition}
Comparing with Section \ref{sec: tridiagonal k toeplitz}, we are interested in the Toeplitz matrix generated by the symbol given by the capacitance matrix $f(z) = f(e^{-\i(\alpha+\i \beta)}) = \widehat{\mathcal{C}}^{\alpha, \beta}$. In the subsequent sections, we will work with complex band functions defined in terms of the frequency $\omega \geq 0$ rather than the eigenvalue $\lambda=\omega^2$.

\begin{figure}[ht!]
    \centering
    \includegraphics[width=0.65\linewidth]{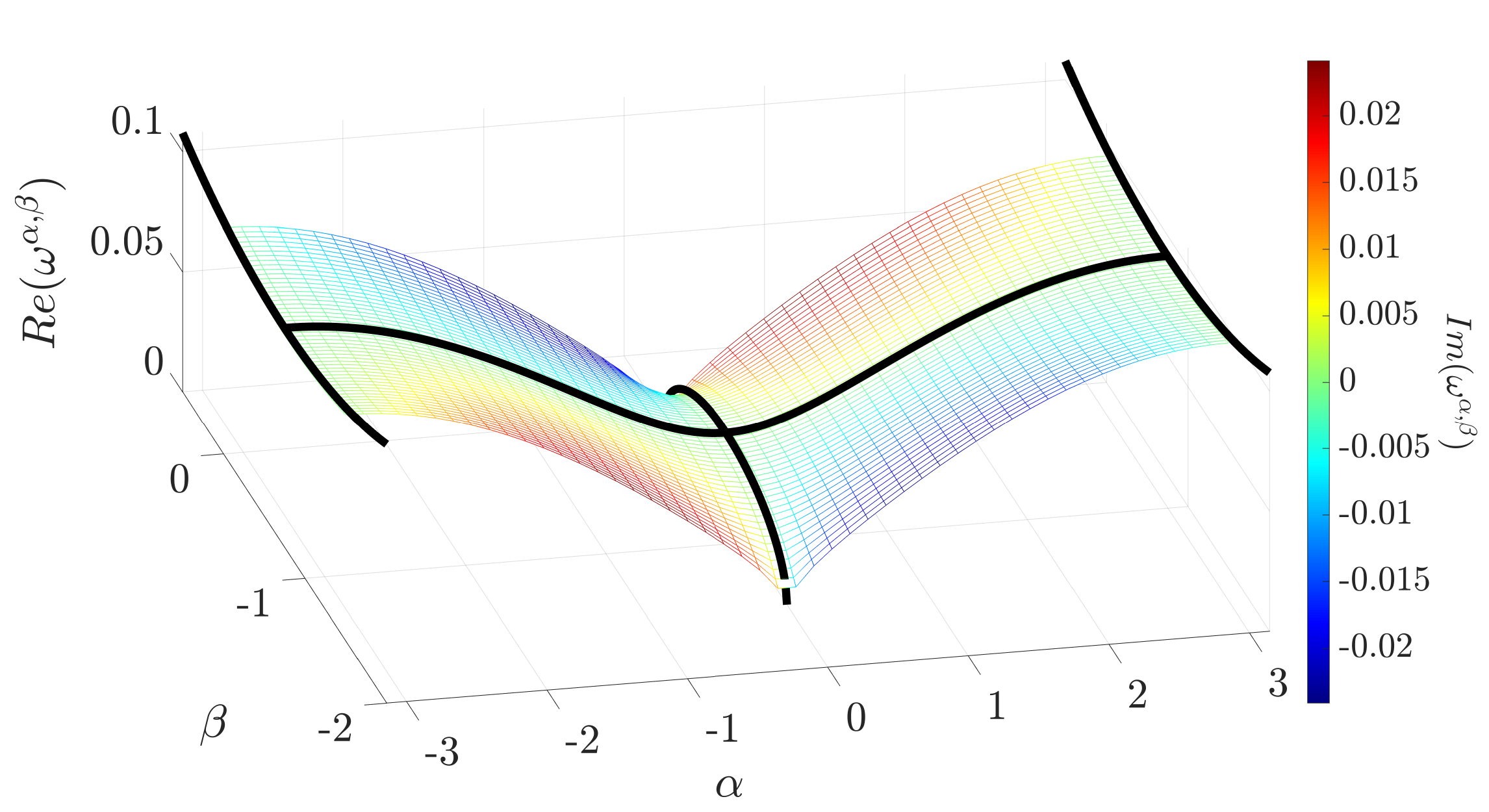}
    \caption{The band function $\omega^{\alpha, \beta}$ as given by Proposition \ref{Prop:  Gauge Capacitance Thm} is plotted as a function of $(\alpha, \beta)$. The colouring in the surface plot illustrates the imaginary part of the band functions $\Im(\omega^{\alpha, \beta})$. The black lines illustrate the surface plot where $\Im(\omega^{\alpha, \beta}) = 0$, which represent admissible resonances to our scattering problem \eqref{eq: wave equation skin effect}. Following Theorem \ref{Thm: alpha and beta fixed} it holds for either $\alpha \in \{0,  \pi\}$ or $\beta = \frac{\gamma}{2}\sum_{i=1}^k \ell_i$.}
    \label{fig: Monomer Band function surface}
\end{figure}

In Figure \ref{fig: Monomer Band function surface}, we illustrate the band function $\omega^{\alpha, \beta}$ in the case of a single repeated subwavelength resonator. Following Theorem \ref{Thm: alpha and beta fixed}, the admissible bands are found at $\alpha \in \{0,  \pi/L\}$ or $\beta = r$. The rate $r$ of non-reciprocity, corresponding to the decay length for frequencies inside the Floquet spectrum $\sigma_\mathrm{open}$, is in this case given by \eqref{eq: non reciproctity rate} and is equal to
\begin{equation}\label{eq: non-reciprocity rate quasiperiodic capacitance}
        r = \frac{1}{2}\log\left( \prod_{i =1}^k \frac{b_i}{c_i} \right)
        =\frac{1}{2}\log\left( \prod_{i =1}^k \frac{\frac{-\gamma}{s_j}\frac{\ell_j}{1-e^{-\gamma \ell_i}}}{
        \frac{\gamma}{s_j}\frac{\ell_j}{1-e^{\gamma \ell_i}}} \right)
        =  \frac{\gamma}{2}\sum_{i=1}^k \ell_i.
\end{equation}

According to Definition \ref{def: reduced band functions}, we distinguish between two types of complex band functions. A band function is the spectral band given by $\omega(\alpha, \beta)$, for fixed ${\beta} = \frac{\gamma}{2}\sum_{i=1}^k \ell_i$. A gap function is given by $\omega({\alpha}, \beta)$ for fixed ${\alpha}\in\{0, \pi/L\}$. In either case, the complex band functions are real-valued solutions to \eqref{eq:omega}.

Two examples of the complex band structure for a non-Hermitian problem are illustrated in Figure \ref{Fig:Gauge_band}. Figure \ref{Fig:Gauge_band} (A) shows the same setup as Figure \ref{fig: Monomer Band function surface} but restricted to the admissible (real-valued) bands. Figure \ref{Fig:Gauge_band} (B) shows the complex band structure for a system of repeated \emph{dimers}, that is, a system with two resonators within the unit cell. Here, we observe a spectral gap of $\sigma_\mathrm{open}$, although $\sigma_\mathrm{wind}$ is not gapped.

Note that the previous definition generalises the band and gap functions to the non-Hermitian case as for Hermitian systems $\gamma = 0$ and thus we recover the definition in the Hermitian case \cite[Definition 2.1]{debruijn2024complexbandstructuresubwavelength}. We also highlight that, because $\beta$ remains constant throughout the Floquet spectrum $\sigma_\mathrm{open}$, the decay lengths of the eigenmodes are uniform across all frequencies within this range.

\begin{figure}[htb]
    \centering
    \subfloat[][Monomer resonator chain with $s_1 = \ell_1 = 0.5$.]%
    {\includegraphics[width=0.45\linewidth]{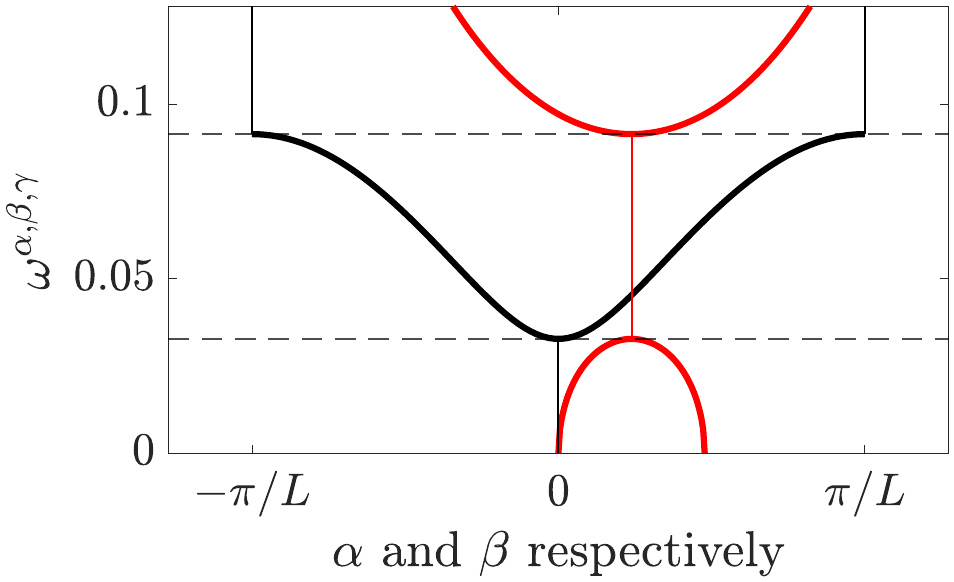}}\quad
    \subfloat[][Dimer resonator chain with $\ell_1 = \ell_2 = 0.25, s_1 = 1,$ and $s_2 = 2$.]%
    {\includegraphics[width=0.45\linewidth]{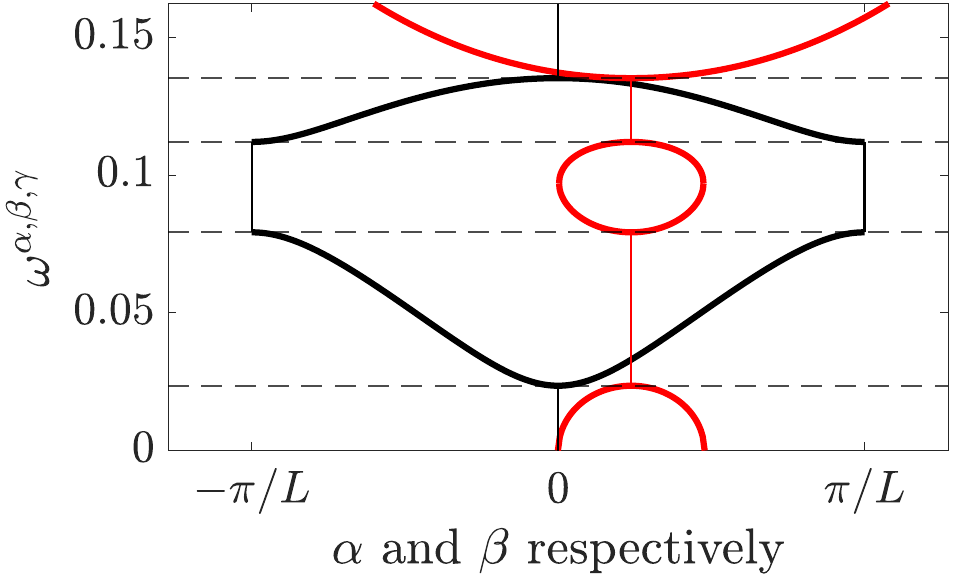}}
    \caption{The band and gap functions plotted as characterised by Proposition \ref{Prop:  Gauge Capacitance Thm}. Figure (A) is exactly the same plot as in Figure \ref{fig: Monomer Band function surface}, but with the admissible bands projected onto the same axis. Qualitatively, the non-Hermitian nature of the problem shifts the origin of $\beta$
    by the rate $r$ of non-reciprocity. Another difference to Hermitian systems is the emergence of a band gap below the first spectral gap. Computation performed for $\gamma = 3$ and $\delta = 10^{-3}$.}
    \label{Fig:Gauge_band}
\end{figure}

\subsection{Defected capacitance matrix}\label{sec:defect_inf}
We will now explore non-Hermitian systems in the presence of defects which break the translational invariance of the problem. Such defects typically create localised modes, in which case Corollary \ref{cor: decay rate outside spectrum} shows that the localisation length can be correctly predicted from the complex band structure. We begin by adapting the approach of  \cite{ErikAnderson2024} to show the existence of localised modes.

 We recall the definition of the complex Floquet transform for quasiperiodic operators (see \cite[Section 5.3]{debruijn2025complexbrillouinzonelocalised}),
 $\F: \ \bigr(\ell^2(\Lambda)\bigr)^k \to \left(L^2(Y)\right)^k$ and its inverse $\I: \ \left(L^2(Y)\right)^k\to ~\bigr(\ell^2(\Lambda)\bigr)^k $, which are given by
	$$\F[\phi](\alpha) := \sum_{\ell\in \Lambda} \phi(\ell)e^{\iu (\alpha+\iu\beta) \ell}, \qquad  \I[\psi](\ell) := \frac{1}{|Y^*|}\int_{Y^*} \psi(\alpha) e^{-\iu (\alpha +\i\beta) \ell} \d \alpha.$$
	Assuming the complex component of the quasimomentum remains constant,  both the Floquet and inverse Floquet transformations remain unaltered from the case of real quasimomentum. We can now define the real-space generalised capacitance coefficients, using the inverse Floquet transform, as
	$$\mathcal{C}^\ell = \I[\widehat{\mathcal{C}}^{\alpha, \beta}](\ell),$$	
	indexed by the real-space variable $\ell \in \Lambda$ and defined for fixed $\beta = r$, where $r$ is given by \eqref{eq: non-reciprocity rate quasiperiodic capacitance}. If $\uf \in \ell^2(\Lambda^k,\Cb^k)$, we define the Laurent operator $\Cf$ on $\ell^2(\Lambda^k,\Cb^k)$, corresponding to the symbol $\widehat{\mathcal{C}}^{\alpha, \beta}$, by
$$\Cf \uf (m) = \sum_{n\in \Lambda} \mathcal{C}^{m-n} \uf(n).$$
The eigenvalue problem reduces to
 \begin{equation}\label{eq:Cf_system}
      \Cf \uf = \omega^2 \uf,
 \end{equation}
and characterises the subwavelength resonances $\omega$ and their corresponding resonant modes $\uf$.

We will use a Green's function approach to describe localisation effects. For eigenfrequencies $\omega^2 \not\in \sigma(\Cf)$, we define the discrete Green's function relative to a point source located at index $j$ by
\begin{equation}
    (\Cf - \omega^2 \Id) \uf = \delta_j,
\end{equation}
where $\delta_j$ represents a source at the index $j$ and is given by the $j$-th standard basis vector. For $\omega^2 \not\in \sigma(\Cf)$ the matrix $(\Cf - \omega^2\Id)$ is invertible, so the Green's function is explicitly defined as,
\begin{equation}\label{eq: discrete Green function}
    \uf = (\Cf-\omega^2\Id)^{-1}\delta_j = \Bigl((\Cf -\omega^2\Id)^{-1}_{i j}\Bigr)_{1\leq i \leq N}.
\end{equation}

The inverses of tridiagonal $k$-Toeplitz matrices are explicit, and closed-form formulas can be found in \cite{InversekToeplitz}. Therefore, the entries of the Green's function are explicit and are determined by the $j$-th column vector of $(\Cf-\omega^2\Id)^{-1}$. As we will demonstrate in the following section, the localisation properties of the discrete Green's function are crucial in the study of defected resonator chains. 

Our main objective is to investigate configurations exhibiting exponentially localised bulk modes. This is achieved by introducing a defect which shifts a resonant frequency outside of the spectrum $\sigma_\mathrm{wind}$. The magnitude of this shift dictates the degree of localisation through the complex band structure. This setup is illustrated in Figure \ref{fig: Defected monomer chain}, where a defect is introduced in the wave speed of a single resonator. 

\begin{figure}[htb]
    \centering
    \begin{adjustbox}{width=\textwidth}
    \begin{tikzpicture}
        \coordinate (x1l) at (1,0);
        \path (x1l) +(1,0) coordinate (x1r);
        \path (x1r) +(0.5,0.7) coordinate (s1);
        \path (x1r) +(1,0) coordinate (x2l);
        \path (x2l) +(1,0) coordinate (x2r);
        \path (x2r) +(0.5,0.7) coordinate (s2);
        \path (x2r) +(1,0) coordinate (x3l);
        \path (x3l) +(1,0) coordinate (x3r);
        \path (x3r) +(.5,0.7) coordinate (s3);
        \path (x3r) +(2,0) coordinate (x4l);
        \path (x4l) +(1,0) coordinate (x4r);
        \path (x4r) +(0.5,0.7) coordinate (s4);
        \path (x4r) +(1,0) coordinate (dots);
        \path (dots) +(1,0) coordinate (x5l);
        \path (x5l) +(1,0) coordinate (x5r);
        \path (x5r) +(1.,0) coordinate (x6l);
        \path (x5r) +(0.5,0.7) coordinate (s5);
        \path (x6l) +(1,0) coordinate (x6r);
        \path (x6r) +(1.,0) coordinate (x7l);
        \path (x6r) +(0.5,0.7) coordinate (s6);
        \path (x7l) +(1,0) coordinate (x7r);
        \path (x7r) +(1.,0) coordinate (x8l);
        \path (x7r) +(0.5,0.7) coordinate (s7);
        \path (x8l) +(1,0) coordinate (x8r);
        \draw[ultra thick] (x1l) -- (x1r);
        \node[anchor=north] (label1) at (x1l) {$x_1^{\iL}$};
        \node[anchor=north] (label1) at (x1r) {$x_1^{\iR}$};
        \node[anchor=south] (label1) at ($(x1l)!0.5!(x1r)$) {$v$};
        \draw[dotted,|-|] ($(x1r)+(0,0.25)$) -- ($(x2l)+(0,0.25)$);
        \draw[ultra thick] (x2l) -- (x2r);
        \node[anchor=north] (label1) at (x2l) {$x_2^{\iL}$};
        \node[anchor=north] (label1) at (x2r) {$x_2^{\iR}$};
        \node[anchor=south] (label1) at ($(x2l)!0.5!(x2r)$) {$v$};
        \draw[dotted,|-|] ($(x2r)+(0,0.25)$) -- ($(x3l)+(0,0.25)$);
        \draw[ultra thick] (x3l) -- (x3r);
        \node[anchor=north] (label1) at (x3l) {$x_3^{\iL}$};
        \node[anchor=north] (label1) at (x3r) {$x_3^{\iR}$};
        \node[anchor=south] (label1) at ($(x3l)!0.5!(x3r)$) {$v$};
        \draw[dotted,|-|] ($(x3r)+(0,0.25)$) -- ($(x4l)+(-1,0.25)$);
        \node[anchor=south] (label1) at ($(x3r) + (1.5,-0.15)$) {$\dots$};
        \node (dots) at (dots) {\dots};
        \draw[red, line width=3pt] (x4l) -- (x4r);
        \node[anchor=north] (label1) at (x4l) {$x_m^{\iL}$};
        \node[anchor=north] (label1) at (x4r) {$x_m^{\iR}$};
        \node[anchor=south] (label1) at ($(x4l)!0.5!(x4r)$) {${\color{red} \Tilde{v}}$};
        \draw[dotted,|-|] ($(x4r)+(0,0.25)$) -- ($(dots)+(-.25,0.25)$);
        \draw[ultra thick] (x5l) -- (x5r);
        \node[anchor=north] (label1) at (x5l) {$x_{N-3}^{\iL}$};
        \node[anchor=north] (label1) at (x5r) {$x_{N-3}^{\iR}$};
        \node[anchor=south] (label1) at ($(x5l)!0.5!(x5r)$) {$v$};
        \draw[dotted,|-|] ($(x5r)+(0,0.25)$) -- ($(x6l)+(0,0.25)$);
        \draw[ultra thick] (x6l) -- (x6r);
        \node[anchor=north] (label1) at (x6l) {$x_{N-2}^{\iL}$};
        \node[anchor=north] (label1) at (x6r) {$x_{N-2}^{\iR}$};
        \node[anchor=south] (label1) at ($(x6l)!0.5!(x6r)$) {$v$};
        \draw[dotted,|-|] ($(x6r)+(0,0.25)$) -- ($(x7l)+(0,0.25)$);
        \draw[ultra thick] (x7l) -- (x7r);
        \node[anchor=north] (label1) at (x7l) {$x_{N-1}^{\iL}$};
        \node[anchor=north] (label1) at (x7r) {$x_{N-1}^{\iR}$};
        \node[anchor=south] (label1) at ($(x7l)!0.5!(x7r)$) {$v$};
        \draw[dotted,|-|] ($(x7r)+(0,0.25)$) -- ($(x8l)+(0,0.25)$);
        \draw[ultra thick] (x8l) -- (x8r);
        \node[anchor=north] (label1) at (x8l) {$x_{N}^{\iL}$};
        \node[anchor=north] (label1) at (x8r) {$x_{N}^{\iR}$};
        \node[anchor=south] (label1) at ($(x8l)!0.5!(x8r)$) {$ v $};
        \node[anchor=north] (label1) at (s1) {$s$};
        \node[anchor=north] (label1) at (s2) {$s$};
        \node[anchor=north] (label1) at (s3) {$s$};
        \node[anchor=north] (label1) at (s4) {$s$};
        \node[anchor=north] (label1) at (s5) {$s$};
        \node[anchor=north] (label1) at (s6) {$s$};
        \node[anchor=north] (label1) at (s7) {$s$};
    \end{tikzpicture}
    \end{adjustbox}
    \caption{A defected monomer chain of $N$ one-dimensional subwavelength resonators, with length $\ell$ and spacing $s$. The wave speed inside the resonators is $v = 1$ whereas in the defected resonator the speed is $\tilde{v} = 1 + \eta.$}
    \label{fig: Defected monomer chain}
\end{figure}
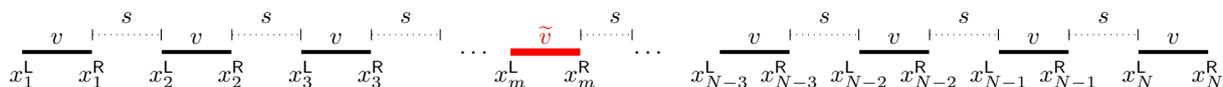
Following \cite{ErikAnderson2024}, defective materials are characterised by the eigenvalue problem
 \begin{equation}\label{eq: defected eigenvalue problem}
     \Bf\Cf\uf = \omega^2 \uf,
 \end{equation}
 where $\Bf$ is diagonal and a compact perturbation of the identity operator. In the case of a single defect at the position $m$ in the chain, $\Bf$ is a doubly-infinite matrix given by
\begin{equation}\label{eq: defect matrix changed wavespeed}
    (\Bf)_{ii} = \begin{cases}
        1, & i \neq m, \\
        1 + \eta, & i = m,
    \end{cases}
\end{equation}
for some constant $\eta \in (-1, \infty)$. Localised modes are associated to eigenfrequencies $\omega^2 \in \sigma(\Bf\Cf)$ such that $\omega^2 \not\in \sigma(\Cf)$.
 
\begin{theorem}\label{thm: green decay equal defect decay}
    Consider a defected structure with a defect in the $m$-th resonator.
    Let $\omega^2 \in \sigma(\Bf\Cf)$ but $\omega^2 \not\in \sigma(\Cf)$, then the discrete Green's function 
    \begin{equation}
        \uf = (\Cf - \omega^2 \Id)^{-1}\delta_m
    \end{equation}
    is an eigenmode of \eqref{eq: defected eigenvalue problem}.
\end{theorem}

\begin{proof}
Note that, since $\Bf$ is diagonal, its inverse is given by
\begin{equation}
    (\Bf)_{ii}^{-1} = \begin{cases}
        1, & i \neq m, \\
        \frac{1}{1 + \eta}, & i = m,
    \end{cases}
\end{equation}
and is well-defined for $\eta > -1$. Note that 
\begin{equation}
    \Bf^{-1} = \Id -\frac{\eta}{1 + \eta}E_{m}, \quad (E_m)_{ij} = \begin{cases}
        1, &i= j=m,\\
        0, &\text{else}.
    \end{cases}
\end{equation}
    A direct computation yields,
    \begin{align}
        \Bf\Cf \uf = \omega^2 \uf
        \Leftrightarrow C \uf = \omega^2 \Bf^{-1}\uf
       \Leftrightarrow \Cf \uf - \omega^2\Id \uf= -\frac{w^2\eta}{1 + \eta}E_{m}
       \Leftrightarrow \uf = -\frac{w^2\eta}{1 + \eta}\bigl(\Cf-\omega^2\Id\bigr)^{-1}\delta_m,
    \end{align}
    which concludes the proof.
\end{proof}

In other words, to study the eigenmodes associated to defect eigenfrequencies, it is enough to study the discrete Green's function at a given defect eigenfrequency $\omega$. Another benefit of this equivalence is that, while $\Bf\Cf$ is not a $k$-Toeplitz operator, we have that $\bigl(\Cf-\omega^2\Id\bigr)$ is a pristine tridiagonal $k$-Toeplitz operator. For $\omega^2 \notin \sigma(\Cf)$, we define
	$$ \mathcal{T}(\omega) =  \left(\begin{smallmatrix} T^0 & T^1 & T^2 & \cdots & T^M   \\
		T^{-1} & T^0 & T^{1} & \cdots & T^{M-1}   \\
		T^{-2} & T^{-1} & T^0 & \cdots & T^{M-2}  \\
		\vdots & \vdots  & \vdots & \ddots & \vdots  \\
		T^{-M} & T^{-(M-1)} & T^{-(M-2)} & \cdots & T^0 \\ \end{smallmatrix}\right), \qquad T^m = -\frac{1}{|Y^*|}\int_{Y^*} e^{\iu \alpha m}\widehat{\mathcal{C}}^{\alpha, \beta}\left(\widehat{\mathcal{C}}^{\alpha, \beta} - \omega^2\Id\right)^{-1}\dx \alpha.$$
The defect eigenfrequencies are characterised by the following result.

 \begin{proposition}\label{prop: defect eigenfrequency}
     Let $\omega_0^2 \not\in\sigma(\Cf)$ and  $\omega_0^2 \in \sigma(\Bf\Cf)$, then $\omega_0$ is an eigenvalue of \eqref{eq: defected eigenvalue problem} if and only if
     \begin{equation}\label{eq: defect eigenfrequency infinit sysmtem}
         \operatorname{det}\bigl(\Id - \mathcal{H}\mathcal{T}(\omega_0)\bigr) = 0,
     \end{equation}
     where $\mathcal{H}$ is the block-diagonal matrix with entries $H_m$.
 \end{proposition}

 \begin{proof}
     The proof is identical to the proof of \cite[Proposition 3.7.]{ErikAnderson2024} but with $\mathcal{C}^\alpha$ replaced by $\mathcal{C}^{\alpha, \beta}$.
 \end{proof}
 In the case of a monomer chain we find the equivalent equation
\begin{equation}\label{eq: monomer gap frequency}
    \frac{\eta}{|Y^*|}\int_{Y^\star}\frac{\lambda^{\alpha, r}_1}{\omega_0^2 - \lambda^{\alpha, r}_1 }\d \alpha = 1.
\end{equation}

In Hermitian systems, it has been observed that defect eigenfrequencies $\omega_0$ are present only when $\eta \geq 0$. A key distinction from non-reciprocal systems lies in the formation of a spectral gap beneath the first spectral band, as illustrated in Figure \ref{Fig:Gauge_band}, which supports a defect mode associated to $\eta < 0$.

\begin{lemma}\label{lemma: existence of def}
    For $\eta < 0$ the system supports a unique defect eigenfrequency in the spectral gap below the spectrum.
    Similarly for $\eta > 0$, there exists a unique defect eigenfrequency above the spectrum. In either case, the defect eigenfrequency is given by
    \begin{equation}
        \omega_0 = \sqrt{\frac{a (\eta+1)^2 + \operatorname{sgn}(\eta)\sqrt{(\eta+1)^2 \left(a^2 \eta^2+8 bc \eta + 4bc\right)}}{2 \eta+1}},
    \end{equation}
    where the coefficients $a,~b,$ and $c$ were defined in \eqref{def: capacitance coefficients}.
\end{lemma}
\begin{proof}
    For a monomer chain, the integral \eqref{eq: monomer gap frequency} may be  explicitly computed. The quasiperiodic capacitance matrix \eqref{eq: quasiperiodic capacitance formula} is now given by the scalar
    \begin{equation}
        \lambda_1^{\alpha, r} =
        \widehat{\mathcal{C}}^{\alpha, r}= -2\sqrt{bc}\cos(\alpha) + a.
    \end{equation} 
    The integral \eqref{eq: monomer gap frequency} evaluates to 
    \begin{equation}\label{eq: integral function of frequency}
        I(\omega_0) = \frac{\eta}{|Y^*|}\int_{Y^\star}\frac{\lambda^{\alpha, r}_1}{\omega_0^2 - \lambda^{\alpha, r}_1 }\d \alpha
         = \eta  \left(\frac{\omega_0^2}{\sqrt{\left(a-\omega_0^2\right)^2-4bc}}-1\right),
    \end{equation}
    where a derivation can be found in  Appendix \ref{appendic: integral evaluation}.
    Since the function \eqref{eq: integral function of frequency} is monotonic within the spectral gaps, there exists a unique $\omega_0^*$ such that
    \begin{equation}\label{eq: defect eigenfrequency condition}
        I(\omega_0^*) = \eta.
    \end{equation}
    By solving \eqref{eq: defect eigenfrequency condition} for $\omega^*$, it follows that  \eqref{eq: monomer gap frequency} is satisfied for 
    \begin{equation}
        \omega_0 = \sqrt{\frac{a (\eta+1)^2 +\operatorname{sgn}(\eta)\sqrt{(\eta+1)^2 \left(a^2 \eta^2+8 bc \eta + 4bc\right)}}{2 \eta+1}},
    \end{equation}
    which completes the proof.
\end{proof}

\subsection{Finite resonator chains.}
Next, we consider finite chains of $N$ resonators. This can be thought of as truncated Toeplitz and Laurent operators, although edge effects will introduce perturbations of the corner entries. The associated finite capacitance matrix has been computed in \cite{FoundationsSkinEffect}.
\begin{definition}
    For $\gamma \in \mathbb{R}\setminus \{0\}$, we define the \emph{(finite) gauge capacitance matrix} $\mathcal{C}^\gamma \in \mathbb{R}^{N\times N}$ by

\begin{equation}\label{capactitance matrix definition}
    \mathcal{C}^\gamma_{i,j} = \begin{cases} \frac{\gamma}{s_1}\frac{\ell_1}{1-e^{-\gamma \ell_1}}, & i = j = 1, \\
    \frac{\gamma}{s_i}\frac{\ell_i}{1-e^{-\gamma \ell_i}} - \frac{\gamma}{s_{i-1}} \frac{\ell_i}{1-e^{\gamma \ell_i}}, & 1 < i = j < N, \\
    -\frac{\gamma}{s_i}\frac{\ell_i}{1-e^{-\gamma \ell_j}}, & 1 \leq i = j-1  \leq N-1, \\
    \frac{\gamma}{s_j}\frac{\ell_i}{1-e^{\gamma \ell_j}}, & 2 < i = j+1 \leq N, \\
    -\frac{\gamma}{s_{N-1}} \frac{\ell_N}{1-e^{\gamma \ell_N}}, & i = j = N,
    \end{cases}
\end{equation}
and all the other entries are zero.  Moreover,  because none of the diagonal and off-diagonal entries are zero, the gauge capacitance matrix is non-degenerate and irreducible.
\end{definition}
It is not hard to see from Definition \ref{capactitance matrix definition} that the gauge capacitance matrix is tridiagonal $k$-Toeplitz with perturbed diagonal corners. Since the perturbed corners consist of compact perturbations, the essential spectrum stays unaffected by them.

\subsubsection{1-periodic Toeplitz matrix}\label{Sec: Defected finite resonator chains}
We now consider the simpler case of $k=1$, in other words, that the Toeplitz matrix is $1$-periodic. In the context of subwavelength resonators, this corresponds to a \emph{monomer} chain with a single resonator inside the unit cell. In this case, explicit formulas of the Green's function are available, which by Theorem \ref{thm: green decay equal defect decay} will allow us to analytically verify the decay estimates achieved by the complex band structure. Defected materials exhibit defect modes; an instance of a defected resonator chain can be seen in Figure \ref{fig: Defected monomer chain}.

The defected capacitance for a system as illustrated in Figure \ref{fig: Defected monomer chain} is given by $B\mathcal{C}^\gamma$, where $B$ is the truncated version of $\Bf$ as defined in \eqref{eq: defect matrix changed wavespeed} and $\mathcal{C}^\gamma$ is defined in \eqref{capactitance matrix definition}.
The resonant frequencies of a finite system comprising a single defect is illustrated in Figure \ref{Fig: Defected monomer spectrum}, and we observe a single resonance frequency associated to the defect.

\begin{figure}[htb]
    \centering
    \subfloat[][$\gamma = 1, n = 50, s_1 = \ell_1 = 0.5$ and $\eta = 1.5$.]%
    {\includegraphics[width=0.45\linewidth]{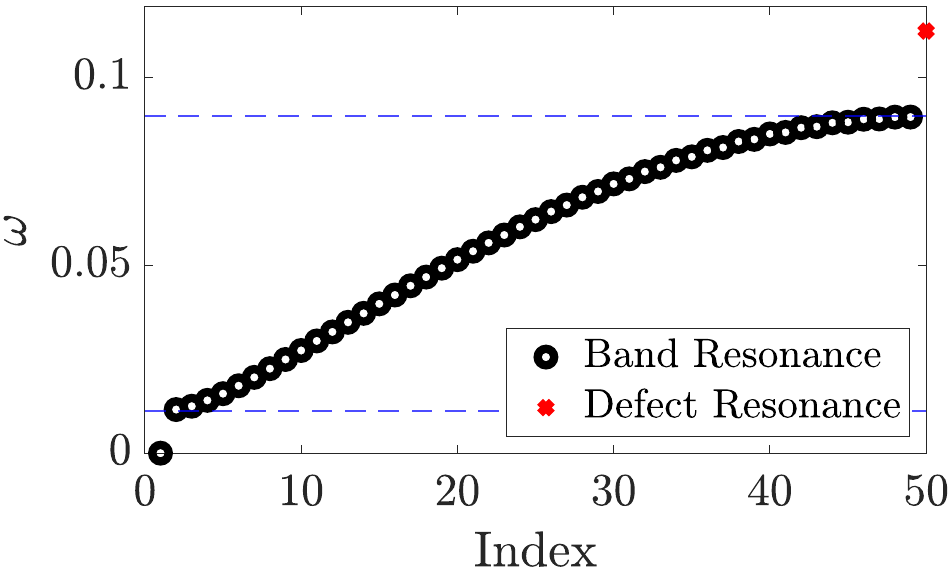}}\quad
    \subfloat[][$\gamma = 1, n = 50, s_1 = \ell_1 = 0.5$ and $\eta = -0.95$.]%
    {\includegraphics[width=0.45\linewidth]{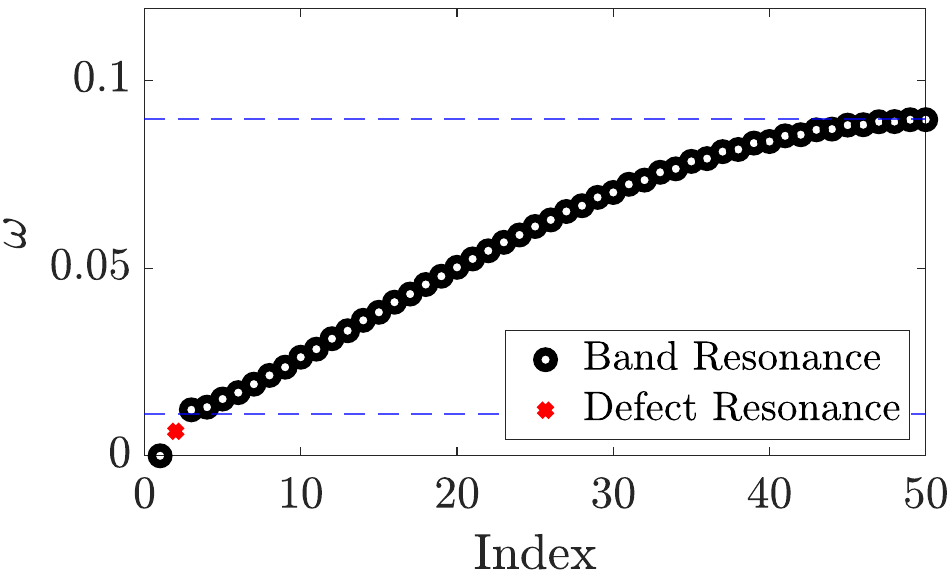}}
    \caption{Spectral plot of a defected structure. Depending on the type of defect, the system supports in accordance with Lemma \ref{lemma: existence of def} a defect eigenfrequency below or above the spectral bands. The dashed blue lines represent the limit of the spectrum given explicitly by \eqref{label: spectrum monomer}.}
    \label{Fig: Defected monomer spectrum}
\end{figure}

For a monomer chain the gauge capacitance matrix \eqref{capactitance matrix definition} is a truncated Toeplitz matrix, with symbol
\begin{equation}\label{eq: symbol for 1 toeplitz matrix}
    f(z) = bz^{-1}+  a + cz \in \C,
\end{equation}
where the entries are given by,
\begin{equation}\label{def: capacitance coefficients}
        a = \frac{\gamma}{s_1}\frac{l_1}{1 - e^{-\gamma l_1}} - \frac{\gamma}{s_1}\frac{l_1}{1 - e^{\gamma l_1}},~
        b = \frac{\gamma}{s_1}  \frac{l_1}{1 - e^{\gamma l_1}},~\text{ and }~
        c = -\frac{\gamma}{s_1}  \frac{l_1}{1 - e^{-\gamma l_1}}.
\end{equation}
For a monomer chain, the rate of non-reciprocity introduced in \eqref{eq: non-reciprocity rate quasiperiodic capacitance} is given by 
    \begin{equation}\label{eq: non reciprocty 1 toeplitz}
        r = \frac{1}{2}\log\left(\frac{b}{c}\right) = \frac{1}{2}\gamma\ell.
    \end{equation}
The next result characterises the open limit of the spectrum of a monomer chain.
\begin{lemma}\label{lemma: open limt monomer}
    It holds that $\omega^2 \not\in \sigma_{\mathrm{open}}$ if and only if $|d| > 1$, where $d := (a - \omega^2)/2\sqrt{bc}$.
\end{lemma}
\begin{proof}
    For a tridiagonal Toeplitz matrix the open limit of the spectrum is given by Theorem \ref{Thm: spectral limit CBZ} that is
    \begin{equation}\label{label: spectrum monomer}
        \bigcup_{\alpha \in Y^*} \sigma\bigl(f(e^{-\i(\alpha + \i r)})\bigr) = \bigcup_{\alpha \in Y^*} \bigl(-2\cos(\alpha)\sqrt{bc} + a\bigr) =  \bigl[-2\sqrt{bc} + a, 2\sqrt{bc} + a\bigr],
    \end{equation}
    where the symbol function $f(z)$ was defined in \eqref{eq: symbol for 1 toeplitz matrix} and $r$ in \eqref{eq: non reciprocty 1 toeplitz}.
    Then a direct computation yields $|d| > 1$ if and only if $\omega^2 \not \in \sigma_{\mathrm{open}}$, which concludes the proof.
\end{proof}

Additionally, from  \cite[Section 3]{ammari2024generalisedbrillouinzonenonreciprocal}, the winding region for a one-dimensional monomer chain is explicitly given by
\begin{equation}\label{eq: monomer winding region}
    \sigma_\text{wind} = \bigl(a-b-c, a + b + c\bigr).
\end{equation}

For the case of 1-periodic Toeplitz matrices, the decay rates of Corollary \ref{cor: decay rate outside spectrum}  can alternatively be established by explicit inversion formulas. For convenience of the reader, we recall some fundamental properties of Chebychev polynomials of the second kind. 
\begin{definition}
    The \emph{Chebychev polynomials of the second kind} form an orthogonal polynomial sequence, defined by the recursion $U_0(x) = 1$, $U_1(x) = 2x$, and $U_{n+1}(x) = 2xU_n(x) - U_{n-1}(x)$.
\end{definition}
For $|x| > 1$, the Chebychev polynomials of second kind admit the following form
\begin{equation}
    U_n(x) = \frac{\sinh{\bigl((n+1)\arccosh{x}\bigr)}}{\sqrt{x^2-1}}
    = \frac{1}{2\sqrt{x^2 -1}} \Bigl( \bigl(x + \sqrt{x^2-1}\bigr)^{n +1} - \bigl(x + \sqrt{x^2-1}\bigr)^{-(n + 1)} \Bigr)
\end{equation}
To put it differently, the Chebychev polynomials of second kind exhibit exponential growth with respect to $n$ when $|x|>1$. To simplify notation, our focus will be on tridiagonal Toeplitz matrices, although analogous findings are applicable to the inverses of broader classes of tridiagonal $k$-Toeplitz matrices \cite{InversekToeplitz, DecayRateInverseTridiag}. In this section, we consistently use the positive root of the $\arccosh{\cdot}$ function.
\begin{theorem}\label{thm: inverse oftridiagonal toeplitz}
    Let $\vect{A}\in \R^{n \times n}$ be a tridiagonal Toeplitz matrix that is irreducible and nonsingular and let $d := a / (2 \sqrt{bc})$, then the inverse is given by
    \begin{equation}\label{eq: inverse of teplitz matrix}
        (\vect{A}^{-1})_{ij} = \begin{cases}
            (-1)^{i+j} \frac{b^{j - i}}{(\sqrt{bc})^{j-i + 1}}\frac{U_{i-1}(d)U_{n-j}(d)}{U_{n}(d)}, \text{ if } i \leq j,\\
            (-1)^{i + j} \frac{c^{i-j}}{(\sqrt{bc})^{i-j+1}}\frac{U_{j-1}(d)U_{n-i}(d)}{U_n(d)}, \text{ if } i > j.
        \end{cases}
    \end{equation}
\end{theorem}

By applying closed formulas for the inverses of Toeplitz matrices, we obtain the decay behaviour of the discrete Green's function \eqref{eq: discrete Green function}. By Theorem \ref{thm: green decay equal defect decay} allows an estimate on the defect eigenmode of a defected resonator chain.

\begin{corollary}\label{Cor: Green up to the defect}
    Let $d = (a - \omega^2)/2\sqrt{bc}$ and $\omega^2 \not\in \sigma_{\mathrm{open}}$. Leading up to the defect site, that is, for $i \leq j$, a defect eigenmode $\vect{u}$ is exponentially localised in $i$, so
    \begin{equation}
        |\vect{u}_i| = \mathcal{O}\bigl( e^{-i(r-\arccosh{d})}\bigl),
    \end{equation}
    with $r = \frac{1}{2}\log\left(\frac{b}{c}\right)$.
\end{corollary}
\begin{proof}
    From Theorem \ref{thm: inverse oftridiagonal toeplitz} it follows that for $i \leq j$, meaning for a mode leading up to the defect,
    \begin{align}
        |\vect{u}_i| &=  \frac{b^{i+j}}{(\sqrt{bc})^{j-i+1}}\frac{U_{i-1}(d)U_{n-j}(d)}{U_n(d)}\\
        &= \frac{U_{n-j}(d)b^j}{U_n(d)(\sqrt{bc})^{j+1}}\frac{b^{-i}U_{i-1}(d)}{(\sqrt{bc})^{i}}\\
        &\leq K e^{i\bigl(\arccosh{d}+\log(\sqrt{bc})-\log(b)\bigr)}.
    \end{align}
    The derived bound originates from the exponential decay of the Chebychev polynomials for $|d|>1$ which by Lemma \ref{lemma: open limt monomer} holds for $\omega^2 \not\in \sigma_{\mathrm{open}}$.
    A direct computation yields $-\log(\sqrt{bc})+\log(b) = \frac{1}{2}\log\left(\frac{b}{c}\right)$, which concludes the proof.
\end{proof}

\begin{corollary}\label{Cor: Green past to the defect}
    Let $d = (a - \omega^2)/2\sqrt{bc}$ and $\omega^2 \not\in \sigma_{\mathrm{open}}$. Past the defect site, that is, for $i > j$, a defect eigenmode $\vect{u}$ is exponentially localised in $i$, so
    \begin{equation}
        |\vect{u}_i| = \mathcal{O}\bigl( e^{-i( r + \arccosh{d}) } \bigr),
    \end{equation}
    with $r = \frac{1}{2}\log\left(\frac{b}{c}\right)$.
\end{corollary}

\begin{proof}
     From Theorem \ref{thm: inverse oftridiagonal toeplitz} it follows that for $i > j$, meaning for a mode past the defect,
    \begin{align}
        |\vect{u}_i| &=  \frac{c^{i-j}}{(\sqrt{bc})^{i-j+1}}\frac{U_{j-1}(d)U_{n-i}(d)}{U_n(d)} \\
        &= \frac{c^{-j}U_{j-1}(d)}{(\sqrt{bc})^{-j+1}U_n(d)}\frac{c^iU_{n-i}(d)}{(\sqrt{bc})^i}\\
        &\leq K e^{-i\bigl( \arccosh{(d)} - \log(c) + \log(\sqrt{bc})\bigr)}.
    \end{align}
    The derived bound originates from the exponential decay of the Chebychev polynomials for $|d|>1$ which by Lemma \ref{lemma: open limt monomer} holds for $\omega^2 \not\in \sigma_{\mathrm{open}}$.
     A direct computation yields $-\log(c) + \log(\sqrt{bc}) =  \frac{1}{2}\log\left(\frac{b}{c}\right)$, which concludes the proof.
\end{proof}

Note that the decay estimates provided in the two previous corollaries are sharp and agree with the estimates presented in Corollaries \ref{cor: decay rate outside spectrum}. Moreover, Corollaries \ref{Cor: Green up to the defect} and \ref{Cor: Green past to the defect} are also valid for $\omega^2\in \sigma_{\mathrm{wind}}^\mathsf{c}$, where corresponding eigenmodes are localised in the bulk rather than at the edge of the resonator chain.
For $\omega^2 \in \sigma_\text{open}$ it holds that $\lvert d \rvert \leq 1$ and thus $\Re\bigl(\arccosh{d}\bigr) = 0$, so the bounds achieved in Corollaries \ref{Cor: Green up to the defect} and \ref{Cor: Green past to the defect} reduce to 
\begin{equation}
    \lvert \vect{u}_i \rvert =\mathcal{O}\bigl(e^{-ir}\bigr), \quad 1 \leq i \leq n,
\end{equation}
which is consistent with Corollary \ref{cor: decay rate spectrum}.

Similarly as in \cite[Section 2.2.5.]{debruijn2024complexbandstructuresubwavelength} the complex band structure can be used to predict the localisation length of eigenmodes in defected resonator chains. In Figure \ref{Greens function against complex band structure} we overlay the decay of the discrete Green's function with the complex band structure and find very close agreement.

\begin{figure}[H]
    \centering
    \subfloat[][$\gamma = 1$]
    {\includegraphics[width=0.45\linewidth]{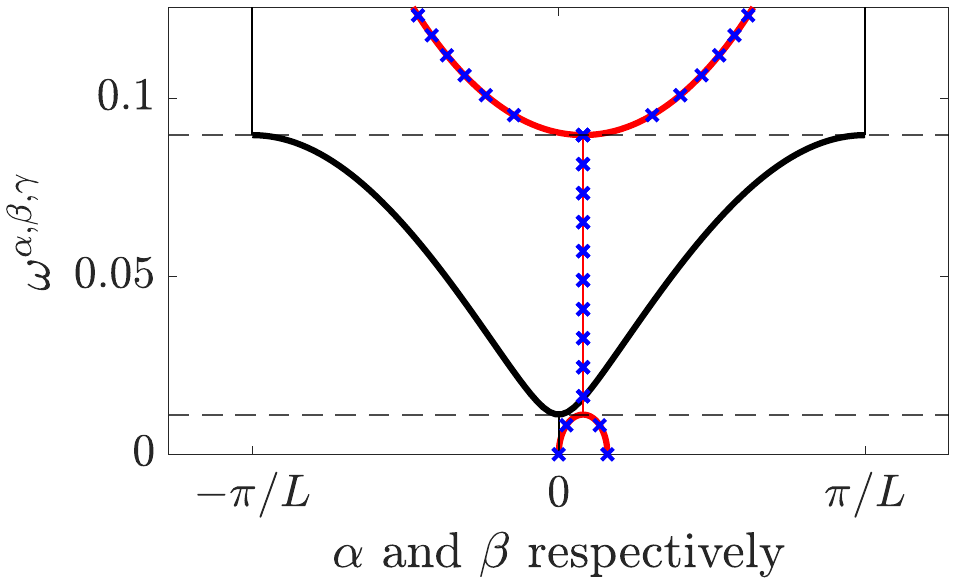}}\quad
    \subfloat[][$\gamma = 3$]
    {\includegraphics[width=0.45\linewidth]{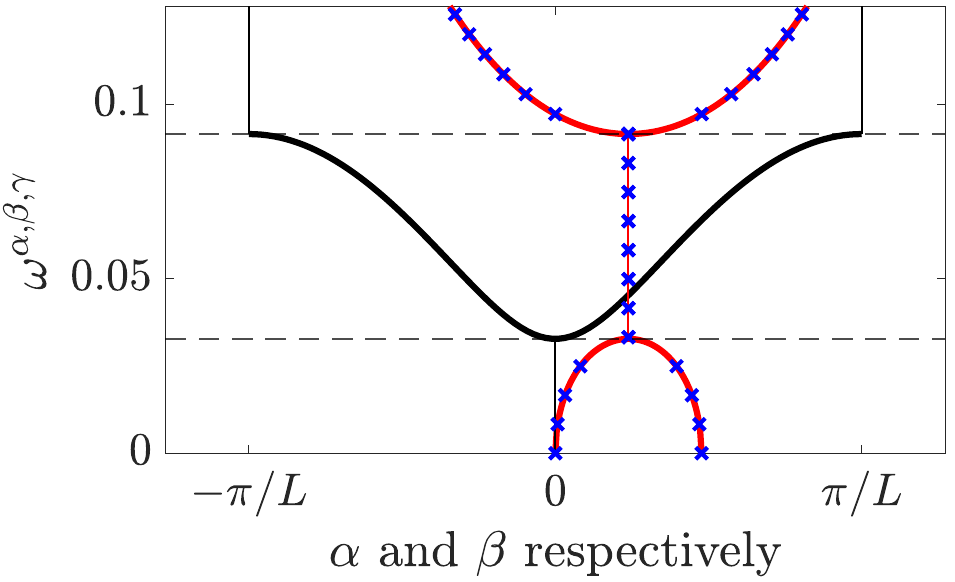}}
    \caption{Gap and band functions for the complex quasiperiodic gauge capacitance matrix. The blue crosses denote the decay bounds of the discrete Green's function as in Corollaries \ref{Cor: Green up to the defect} and \ref{Cor: Green past to the defect}. For $\omega^2\in \sigma_\text{open}$ the Green's function is computed numerically via the pseudoinverse.
    Computation performed for $s_i = \ell_i = 0.5$. 
    }
    \label{Greens function against complex band structure}
\end{figure}

\subsubsection{Polymer chains}
The complex band structure may be used for general $k$-periodic Toeplitz matrices and provides an effective and relatively simple method for modelling localisation phenomena in finite periodic polymer chains of resonators. The benefit of the complex band structure is the following: Consider a $k$-periodic resonator chain comprising $N$ unit cells. Calculating the decay of the Green's function would require inverting a matrix of dimension $Nk$. However, the complex band structure is entirely defined by the $k$ eigenvalues of the symbol, and according to Section \ref{Sec: Pseudospectrum}, we know that this will be an exponentially good decay estimate. This will be particularly important when working on relatively large systems where $N$ is large.

\begin{figure}[ht]
    \centering
    \subfloat[][$\eta = 1.5$]%
    {\includegraphics[width=0.45\linewidth]{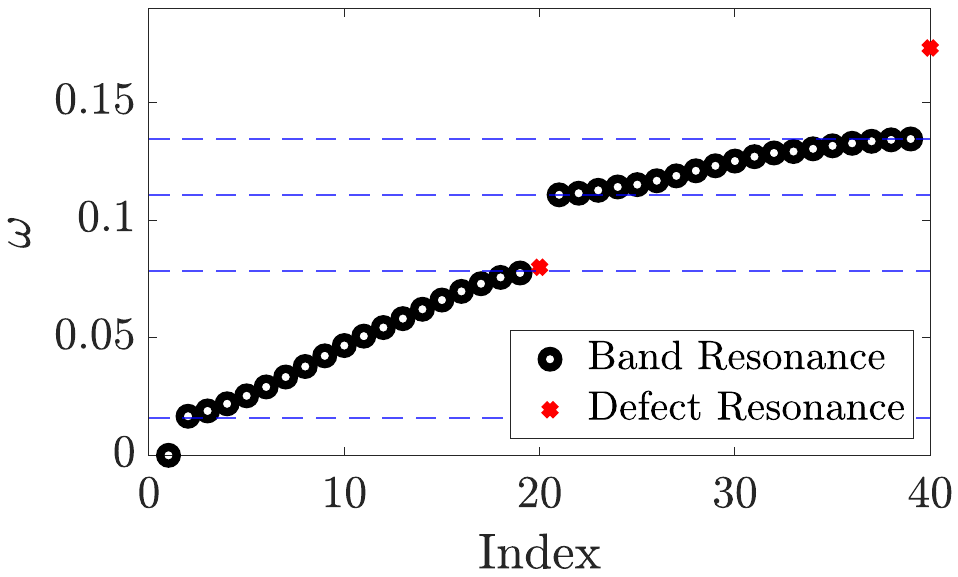}}\quad
    \subfloat[][$\eta = -0.95$]%
    {\includegraphics[width=0.45\linewidth]{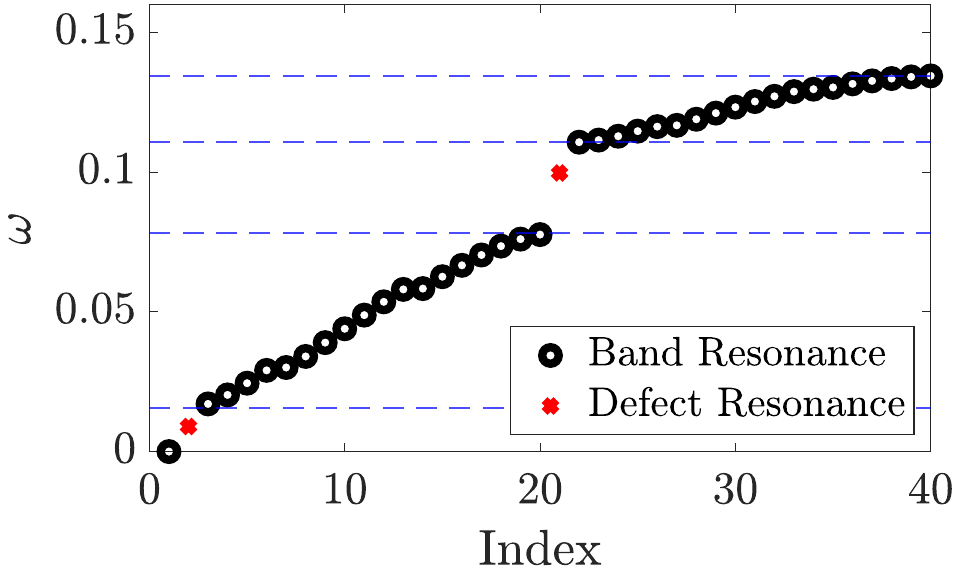}}
    \caption{The eigenfrequencies of a finite dimer resonator chain. The dashed lines illustrate the spectrum of an infinite chain. The two defect modes that lie in the spectral gap are highlighted in red.  Computation performed for $s_1 = 1, s_2 = 2, \ell_1 = \ell_2 = 0.25, \gamma = 3$ and $N = 40$. }
    \label{fig: dimer defect}
\end{figure}

\begin{figure}[ht]
    \centering
    \subfloat[][$\gamma = 1$]%
    {\includegraphics[width=0.45\linewidth]{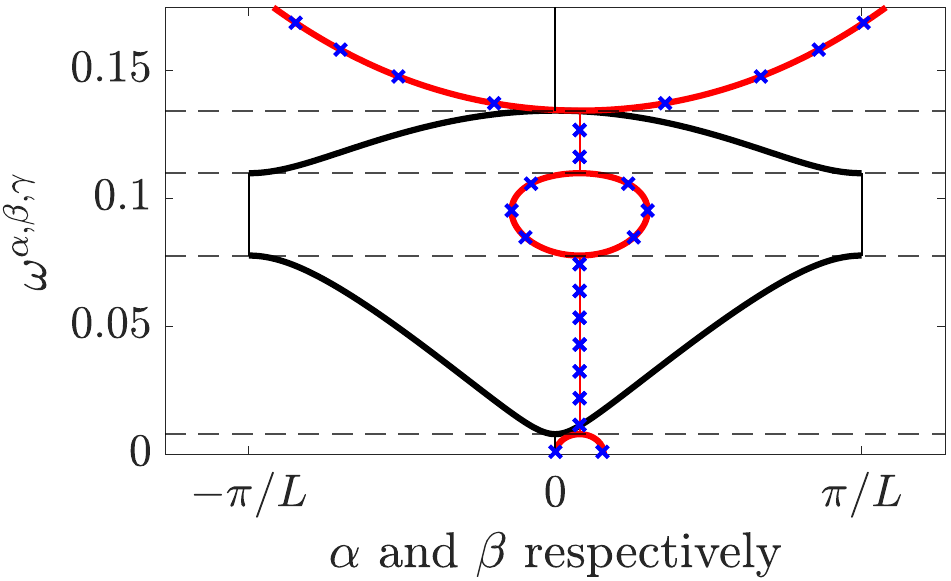}}\quad
    \subfloat[][$\gamma = 3$]%
    {\includegraphics[width=0.45\linewidth]{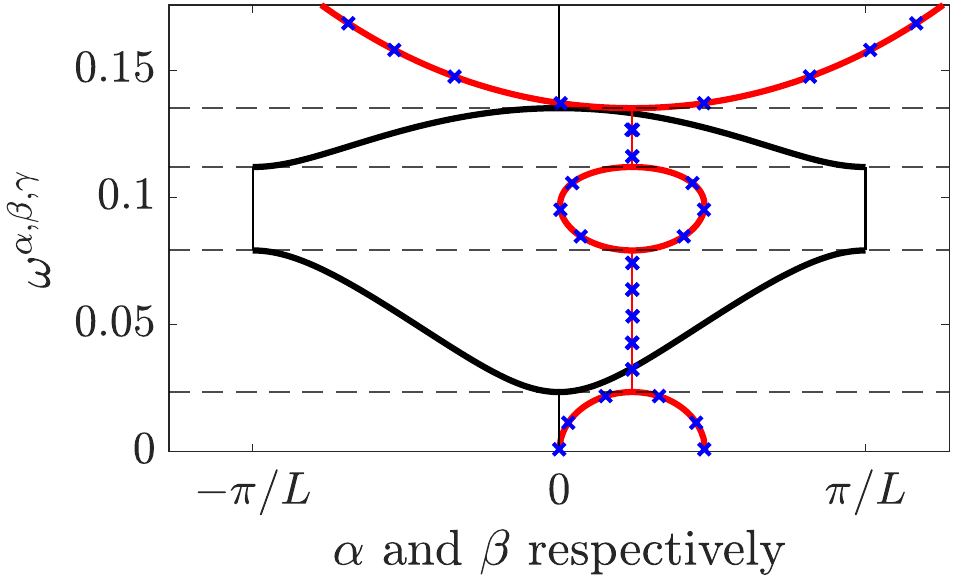}}
    \caption{Gap and band functions for the complex quasiperiodic gauge capacitance matrix. The blue crosses denote the decay bounds of the discrete Green's function computed numerically. For $\omega^2\in \sigma_\text{open}$ the Green's function is computed numerically via the pseudoinverse. Computation performed for $s_1 = 1, s_2 = 2, \ell_1 = \ell_2 = 0.25,$. }
    \label{fig: Greens function against complex band structure Dimer}
\end{figure}

As shown in Figure \ref{fig: dimer defect}, we may obtain modes with frequencies outside the Floquet spectrum by introducing defects to the dimer system. Figure \ref{fig: Greens function against complex band structure Dimer} shows that the decay length of the discrete Green's function, which corresponds to simulating finite structures with defects at varied frequencies, aligns remarkably well with the complex band structure. This alignment confirms the reliability of the complex band structure in predicting localisation phenomena within finite non-Hermitian polymer chains.

\section{Non-Hermitian Tight-Binding Hamiltonian}\label{sec: Non-Hermitian Hamiltonian}
In this section, we further illustrate the wide applicability of the complex band structure for tridiagonal $k$-Toeplitz operators. We proceed to investigate defect modes in the non-Hermitian tight binding Hamiltonian with non-reciprocal coupling introduced by Hatano and Nelson \cite{PhysRevB.58.8384, PhysRevLett.77.570, PhysRevB.111.035109}. For an unperturbed structure, the non-Hermitian Hamiltonian is given by 
\begin{equation}\label{eq: Hamiltonian}
    H\psi_n = V_n \psi_n - e^{-\gamma}\psi_{n-1} - e^{\gamma}\psi_{n+1},
\end{equation}
where $V_n: \Z \to \R$ is some potential function and the non-reciprocity is given by $\gamma > 0$.
For a constant potential $V_n = v$, the Hamiltonian \eqref{eq: Hamiltonian} is a truncated Toeplitz matrix with symbol,
\begin{equation}\label{eq: symbol non-Hermitian Hamiltonian}
    f(z) = v - e^{-\gamma} z - e^{\gamma}z^{-1}.
\end{equation}
For eigenvalues within the open spectral band, as specified in \eqref{eq. spectral theorem limit}, Theorem \ref{Thm: alpha and beta fixed} indicates the non-reciprocity rate given by \eqref{eq: non reciproctity rate} is fixed to,
\begin{equation}\label{eq: non-reciprocity rate non-Hermitian Hemiltonian}
    r = \frac{1}{2}\log\Bigl(\prod_{i=1}^k \frac{b_i}{c_i}\Bigr) = \frac{1}{2}\log\Bigl(\frac{e^{\gamma}}{e^{-\gamma}}\Bigr) =  \gamma.
\end{equation}
By Corollary \ref{cor: decay rate spectrum} the eigenvectors of the Toeplitz operator with symbol \eqref{eq: symbol non-Hermitian Hamiltonian} are exponentially localised with rate equal to the rate of non-reciprocity. 
As shown in Section \ref{Sec: Pseudospectrum}, truncated eigenvectors of the Toeplitz operator may be used to construct exponentially good pseudoeigenvectors for the finite Toeplitz matrix.
As the symbol function \eqref{eq: symbol non-Hermitian Hamiltonian} is scalar valued, a direct computation yields
\begin{equation}\label{eq:simplified symbol}
    f\bigl(e^{\i(\alpha + \i \beta)}\bigr) = v-e^{-\gamma}e^{\i(\alpha + \i \beta)} - e^{\gamma}e^{-\i(\alpha + \i\beta)} = v-2\cos(\alpha)\cosh(\gamma + \beta).
\end{equation}
The open limit of the spectrum is given by Lemma \ref{lemma: open limt monomer}, that is 
\begin{equation}
    \sigma_\text{open} = [v - 2, v + 2 ].
\end{equation}

For a defect located at index $j$ in the on-site potential, the potential is described by 
\begin{equation}
    V_n = v + d\delta_{n,j},
\end{equation}
with $d\in \R$ defining the defect's magnitude. Variations in $d$ result in distinct localisation characteristics. As detailed in \cite{PhysRevB.111.035109}, when the defect magnitude is
\begin{equation}\label{eq: defect edge of winding region}
    d = \pm2 \sinh(\gamma),
\end{equation}
the Toeplitz matrix features an eigenvalue $\lambda \in \sigma_\text{det}$, positioned on the boundary of the winding region.  According to Corollary \ref{cor: decay rate outside spectrum}, the exponential decay rates are $r - \arccosh{\frac{v-\lambda}{2}}$ and $r + \arccosh{\frac{v-\lambda}{2}}$ on the left and right sides of the defect, respectively.

\begin{figure}[tb]
    \centering
   \includegraphics[width=0.85\linewidth]{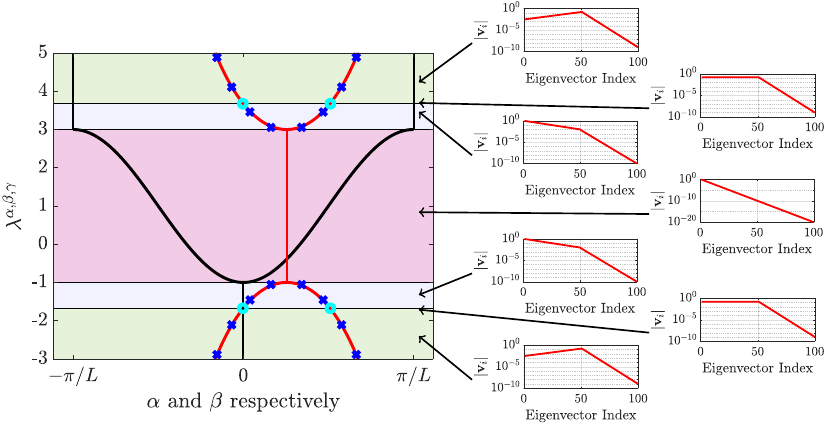}
    \caption{The band and gap functions for the symbol function \eqref{eq:simplified symbol}. The dark blue crosses denote defects introduced into the structure represent the eigenvalues and decay of the eigenvectors of a defected Hamiltonian, simulated for a variety of defect sizes. The light blue circle are the decay lengths of a defect of size \eqref{eq: defect edge of winding region} and are situated at the edge of the winding region. The qualitative behaviour of the eigenvectors is the same as in Figure \ref{fig: Spectral plot zones}.}
    \label{fig: Band functions non-Hermitian Hamiltonian}
\end{figure}

This transition from skin to bulk localisation,  illustrated in Figure \ref{fig: Band functions non-Hermitian Hamiltonian}, can be explained by the complex band structure with the following reasoning. 
For skin localisation to occur, it must hold that the eigenmodes are exponentially decaying leading up to the defect and past the defect, that is 
\begin{equation}
    \begin{cases}
        r + \arccosh{\frac{v-w}{2}} > 0,\\
        r - \arccosh{\frac{v-w}{2}} > 0,
    \end{cases}
\end{equation}
which is satisfied for frequencies $w \in \bigl(v-2\cosh{(r)}, v + 2\cosh{(r)}\bigr)$, which is consistent with the values presented in \cite{PhysRevB.111.035109} and coincides with the winding region given in \eqref{eq: monomer winding region}. For frequencies outside this interval, the exponential decay rate is $r - \arccosh{\frac{v-w}{2}} \leq 0$, in other words, the eigenmode is exponentially growing leading up to the defect, and the eigenmode is localised in the bulk. 

\section{Concluding remarks}\label{Sec: concluding remarks}
Our research illustrates that the complex band structure presents a universal method for analysing perturbed tridiagonal $k$-Toeplitz operators, resulting in a detailed characterisation of the localisation properties of its eigenvectors. We have demonstrated that the complex band structure serves as an effective basis for generating pseudoeigenvectors with exponential accuracy for finite systems. This approach facilitates the exploration of defected non-Hermitian materials, leading to the identification of novel phenomena, including non-reciprocal interface modes and skin localisation. The methodology for tridiagonal $k$-Toeplitz operators is general, and we demonstrated multiple applications in subwavelength physics and quantum mechanics. In future work, we seek to extend the complex band structure to a wider class of pseudo-Hermitian operators such as $m$-banded Toeplitz matrices for $m > 3$, allowing one to treat systems with long-range interactions.

\section{Data availability} \label{Sec: Data availability}
The \texttt{Matlab} code for the numerical experiments developed in this work is openly available in the following repository:
\url{https://github.com/yannick2305/Non-Hermitian-Localisation}.

\appendix
\section{Integral Evaluation} \label{appendic: integral evaluation}
We seek to evaluate the integral,
    \begin{align}
         \frac{\eta}{|Y^*|}\int_{Y^\star}\frac{\lambda^{\alpha, r}_1}{\omega_0^2 - \lambda^{\alpha, r}_1 }\d \alpha
         &= \frac{\eta}{2\pi} \int_{-\pi}^\pi   \frac{a -\sqrt{bc}\cos(\alpha) }{w_0^2 -a + \sqrt{bc}\cos(\alpha) }\d \alpha,
    \end{align}
    which is an integral of the form
    \begin{equation}\label{eq: integral cos fraction}
        \int_{-\pi}^\pi \frac{A + B\cos(x)}{C - B\cos(x)}\d x = 2\int_{0}^\pi \frac{A + B\cos(x)}{C - B\cos(x)}\d x = 2\frac{A}{B}\int_0^\pi\frac{\d x}{\frac{C}{B}-\cos(x)} + 2\frac{B}{C}\int_0^\pi\frac{\cos(x)\d x}{1-\frac{B}{C}\cos(x)}.
    \end{equation}
    The integrals can now be evaluated as follows. By \cite[3.613 (1)]{gradshteyn2014tables}
    \begin{equation}
        \int_0^\pi \frac{\cos(x)}{1-B/C\cos(x)}\d x = \frac{\pi}{\sqrt{1-(B/C)^2}}\left(\frac{\sqrt{1-(B/C)^2}-1}{B/C}\right).
    \end{equation}
    Also, by \cite[3.661 (4)]{gradshteyn2014tables},
    \begin{equation}
        \int_0^\pi \frac{\d x}{C/B-\cos(x)} = \frac{\pi}{\sqrt{(C/B)^2 -1}}.
    \end{equation}
    As a consequence, the integral \eqref{eq: integral cos fraction} evaluates to
    \begin{equation}
        \frac{2\pi}{\sqrt{C^2-B^2}}\left(A -  \sqrt{C^2 - B^2}-C\right) = 2\pi \left(\frac{A-C}{\sqrt{C^2-B^2}}-1 \right).
    \end{equation}
    Substituting the values $A = a, B = -2\sqrt{bc}$ and $ C = w^2_0-a$ yields the desired result.

\bibliographystyle{abbrv}
\bibliography{bibliography_master}{}

\begin{thebibliography}{10}

\bibitem{10.1121/10.0022535}
H.~B. Al~Ba'ba'a.
\newblock Brillouin-zone definition in non-reciprocal willis monatomic lattices.
\newblock {\em JASA Express Letters}, 3(12):120001, 12 2023.

\bibitem{FoundationsSkinEffect}
H.~Ammari, S.~Barandun, J.~Cao, B.~Davies, and E.~O. Hiltunen.
\newblock Mathematical foundations of the non-hermitian skin effect.
\newblock {\em Archive for Rational Mechanics and Analysis}, 248(3):33, 2024.

\bibitem{ammari2024spectra}
H.~Ammari, S.~Barandun, Y.~De~Bruijn, P.~Liu, and C.~Thalhammer.
\newblock Spectra and pseudo-spectra of tridiagonal k-toeplitz matrices and the topological origin of the non-hermitian skin effect.
\newblock {\em Journal of Physics A: Mathematical and Theoretical}, 58(20):205201, 05 2025.

\bibitem{ammari2024generalisedbrillouinzonenonreciprocal}
H.~Ammari, S.~Barandun, P.~Liu, and A.~Uhlmann.
\newblock Generalised brillouin zone for non-reciprocal systems.
\newblock {\em Proc. R. Soc. A}, 481:20240643, 2025.

\bibitem{ErikAnderson2024}
H.~Ammari, B.~Davies, and E.~O. Hiltunen.
\newblock Anderson localization in the subwavelength regime.
\newblock {\em Communications in Mathematical Physics}, 405(1):1, 2024.

\bibitem{ammari2024functional}
H.~Ammari, B.~Davies, and E.~Orvehed~Hiltunen.
\newblock Functional analytic methods for discrete approximations of subwavelength resonator systems.
\newblock {\em Pure and Applied Analysis}, 6(3):873--939, 2024.

\bibitem{ammari.fitzpatrick.ea2018Mathematical}
H.~Ammari, B.~Fitzpatrick, H.~Kang, M.~Ruiz, S.~Yu, and H.~Zhang.
\newblock {\em Mathematical and Computational Methods in Photonics and Phononics}, volume 235 of {\em Mathematical Surveys and Monographs}.
\newblock {American Mathematical Society, Providence, RI}.

\bibitem{ashida.gong.ea2020NonHermitian}
Y.~Ashida, Z.~Gong, and M.~Ueda.
\newblock Non-{{Hermitian}} physics.
\newblock 69(3):249--435.

\bibitem{PhysRevLett.80.5243}
C.~M. Bender and S.~Boettcher.
\newblock Real spectra in non-hermitian hamiltonians having pt symmetry.
\newblock {\em Phys. Rev. Lett.}, 80:5243--5246, 6 1998.

\bibitem{PhysRevLett.124.056802}
D.~S. Borgnia, A.~J. Kruchkov, and R.-J. Slager.
\newblock Non-hermitian boundary modes and topology.
\newblock {\em Phys. Rev. Lett.}, 124:056802, 2 2020.

\bibitem{bottcher.silbermann1999Introduction}
A.~B\"ottcher and B.~Silbermann.
\newblock {\em Introduction to Large Truncated {{Toeplitz}} Matrices}.
\newblock Universitext. {Springer-Verlag, New York}.

\bibitem{PhysRevB.56.R4333}
P.~W. Brouwer, P.~G. Silvestrov, and C.~W.~J. Beenakker.
\newblock Theory of directed localization in one dimension.
\newblock {\em Phys. Rev. B}, 56:R4333--R4335, 8 1997.

\bibitem{debruijn2025complexbrillouinzonelocalised}
Y.~D. Bruijn and E.~O. Hiltunen.
\newblock Complex brillouin zone for localised modes in hermitian and non-hermitian problems, 2025.

\bibitem{InversekToeplitz}
C.~da~Fonseca and J.~Petronilho.
\newblock Explicit inverse of a tridiagonal k-toeplitz matrix.
\newblock {\em Numerische Mathematik}, 100(3):457--482, 2005.

\bibitem{PhysRevB.111.035109}
B.~Davies, S.~Barandun, E.~O. Hiltunen, R.~V. Craster, and H.~Ammari.
\newblock Two-scale effective model for defect-induced localization transitions in non-hermitian systems.
\newblock {\em Phys. Rev. B}, 111:035109, 2025.

\bibitem{debruijn2024complexbandstructuresubwavelength}
Y.~De~Bruijn and E.~O. Hiltunen.
\newblock Complex band structure for subwavelength evanescent waves.
\newblock {\em Studies in Applied Mathematics}, 154(2):e70022, 2025.

\bibitem{feppon.cheng.ea2023Subwavelength}
F.~Feppon, Z.~Cheng, and H.~Ammari.
\newblock Subwavelength resonances in one-dimensional high-contrast acoustic media.
\newblock 83(2):625--665.

\bibitem{Fern_ndez_2022}
F.~M. Fernández.
\newblock On a class of non-hermitian hamiltonians with tridiagonal matrix representation.
\newblock {\em Annals of Physics}, 443:169008, Aug. 2022.

\bibitem{ghatak.brandenbourger.ea2020Observation}
A.~Ghatak, M.~Brandenbourger, J.~Van~Wezel, and C.~Coulais.
\newblock Observation of non-{{Hermitian}} topology and its bulk\textendash edge correspondence in an active mechanical metamaterial.
\newblock 117(47):29561--29568.

\bibitem{gohberg1993basic}
I.~Gohberg, M.~A. Kaashoek, and S.~Goldberg.
\newblock {\em Classes of Linear Operators}.
\newblock Birkh{\"a}user Basel, 1993.

\bibitem{gradshteyn2014tables}
I.~S. Gradshteyn and I.~M. Ryzhik.
\newblock {\em Table of Integrals, Series, and Products}.
\newblock Academic Press, San Diego, CA, 7th edition, 2014.

\bibitem{PhysRevLett.77.570}
N.~Hatano and D.~R. Nelson.
\newblock Localization transitions in non-hermitian quantum mechanics.
\newblock {\em Phys. Rev. Lett.}, 77:570--573, 7 1996.

\bibitem{PhysRevB.58.8384}
N.~Hatano and D.~R. Nelson.
\newblock Non-hermitian delocalization and eigenfunctions.
\newblock {\em Phys. Rev. B}, 58:8384--8390, 10 1998.

\bibitem{NonHermitianCloaking}
S.~Jana and L.~Sirota.
\newblock Invisible tunneling through non-hermitian barriers in nonreciprocal lattices.
\newblock {\em Phys. Rev. B}, 111:L100301, 2025.

\bibitem{PhysRevResearch.1.023013}
S.~Longhi.
\newblock Probing non-hermitian skin effect and non-bloch phase transitions.
\newblock {\em Phys. Rev. Res.}, 1:023013, 9 2019.

\bibitem{longhi}
S.~Longhi.
\newblock Non-hermitian skin effect beyond the tight-binding models.
\newblock {\em Phys. Rev. B}, 104:125109, 2021.

\bibitem{DecayRateInverseTridiag}
R.~Nabben.
\newblock Decay rates of the inverse of nonsymmetric tridiagonal and band matrices.
\newblock {\em SIAM Journal on Matrix Analysis and Applications}, 20(3):820--837, 1999.

\bibitem{PhysRevLett.124.086801}
N.~Okuma, K.~Kawabata, K.~Shiozaki, and M.~Sato.
\newblock Topological origin of non-hermitian skin effects.
\newblock {\em Phys. Rev. Lett.}, 124:086801, 3 2020.

\bibitem{Non-HermitianTopologicalPhenomenaAReview}
N.~Okuma and M.~Sato.
\newblock Non-hermitian topological phenomena: A review.
\newblock {\em Annual Review of Condensed Matter Physics}, 14:83--107, 2023.

\bibitem{REICHEL1992153}
L.~Reichel and L.~N. Trefethen.
\newblock Eigenvalues and pseudo-eigenvalues of toeplitz matrices.
\newblock {\em Linear Algebra and its Applications}, 162-164:153--185, 1992.

\bibitem{trefethen.embree2005Spectra}
L.~N. Trefethen and M.~Embree.
\newblock {\em Spectra and Pseudospectra}.
\newblock {Princeton University Press, Princeton, NJ}.

\bibitem{PhysRevLett.125.226402}
Z.~Yang, K.~Zhang, C.~Fang, and J.~Hu.
\newblock Non-hermitian bulk-boundary correspondence and auxiliary generalized brillouin zone theory.
\newblock {\em Phys. Rev. Lett.}, 125:226402, 11 2020.

\bibitem{PhysRevLett.121.086803}
S.~Yao and Z.~Wang.
\newblock Edge states and topological invariants of non-hermitian systems.
\newblock {\em Phys. Rev. Lett.}, 121:086803, 8 2018.

\bibitem{PhysRevLett.125.126402}
K.~Zhang, Z.~Yang, and C.~Fang.
\newblock Correspondence between winding numbers and skin modes in non-hermitian systems.
\newblock {\em Phys. Rev. Lett.}, 125:126402, 9 2020.

\bibitem{Zhang_2024}
X.~Zhang, X.~Song, S.~Zhang, T.~Zhang, Y.~Liao, X.~Cai, and J.~Li.
\newblock Solvable non-hermitian skin effects and real-space exceptional points: non-hermitian generalized bloch theorem.
\newblock {\em Journal of Physics A: Mathematical and Theoretical}, 57(12):125001, 3 2024.

\end{thebibliography}
\end{document}